\documentclass[reqno]{amsart}

\newif\ifsiamjnl
\ifcsname newsiamthm\endcsname%
\siamjnltrue\else\siamjnlfalse%
\fi
\ifsiamjnl
\usepackage{amsmath,amssymb}
\usepackage{comment}
\usepackage{enumitem}
\setlist[enumerate]{leftmargin=.5in}
\setlist[itemize]{leftmargin=.5in}
\else
\usepackage[citebordercolor={0.7 0.7 0.7},linkbordercolor={0.7 0.7 1}]{hyperref}
\usepackage{comment,graphicx,color}
\fi
\usepackage[all]{xy}
\excludecomment{maximacode}
\excludecomment{octavecode}
\excludecomment{arxivabstract}
\includecomment{arma}
\ifcsname citenum\endcsname
\fi
\input{ltbunicode.sty}

\ifcsname showkeystrue\endcsname
\usepackage[notcite,notref]{showkeys}

\fi

\DeclareGraphicsExtensions{.png,jpg,jpeg,.mps,.pdfmac,.pdffig,.spdf,.pdf}
\newcommand{\T}{{\bf T}}
\newcommand{\R}{{\bf R}}

\newcommand{\Z}{{\bf Z}}

\newcommand{\set}[1]{\left\{#1\right\}}

\ifsiamjnl
\newsiamremark{remark}{Remark}
\newsiamremark{example}{Example}
\else
\newtheorem{theorem}{Theorem}[section]
\newtheorem{corollary}{Corollary}[section]

\newtheorem{lemma}{Lemma}[section]
\newtheorem{proposition}{Proposition}[section]

\theoremstyle{definition}
\newtheorem{definition}{Definition}[section]
\theoremstyle{remark}

\newtheorem{remark}{Remark}[section]

\fi

\newcommand{\safeinput}[1]{\IfFileExists{#1}{{\catcode`\&=0\input{#1}}}{\message{File
        #1 does not exists. Skipping.}}}
\def\ltxfigure#1#2#3{\resizebox{#2}{#3}{\input{#1}}}

\newcommand{\sphere}[1]{{\mathbf{S}^{#1}}}

\newcommand{\D}[1]{\mathrm{d}#1}
\newcommand{\mean}[1]{\overline{#1}}

\newcommand{\textem}[1]{{\em #1}}
\newcommand{\defn}[1]{\textem{#1}}

\ifcsname nhpaper\endcsname
\else
\fi

\newcommand{\temp}[1]{2{\mathrm K}}

\newcommand{\ip}[2]{\langle\hspace{-0.5ex}\langle #1, #2 \rangle\hspace{-0.5ex}\rangle}
\newcommand{\kp}[2]{\langle #1, #2 \rangle}
\newcommand{\ddt}[2][\ ]{\frac{\D{#1}}{\D{#2}}}
\newcommand{\didi}[2][\ ]{\frac{∂#1}{∂#2}}
\newcommand{\didin}[3]{\frac{∂^{#1}{#2}}{∂{#3}^{#1}}}
\newcommand{\cotangent}{T^*}
\newcommand{\metaref}[3]{{\ifnum#1=0(\fi}{#3}\ref{#2}{\ifnum#1=0)\fi}}
\renewcommand{\eqref}[2][0]{\metaref{#1}{#2}{eq.~}}
\newcommand{\defref}[2][0]{\metaref{#1}{#2}{def.~}}

\newcommand{\pb}[2]{\left\{ #1,#2 \right\}}

\newcommand{\DD}[2]{\mathrm{D}^{(#1)}{#2}}
\newcommand{\tempH}[2]{\DD{#1}{#2}}
\newcommand{\nthderivative}[2]{\mathrm{d}^{#1}{#2}}
\newcommand{\thermostateqm}[1]{\mathcal{#1}}
\newcommand{\closure}[1]{\overline{#1}}
\newcommand{\tI}[0]{\tilde{I}}
\newcommand{\hI}[0]{\hat{I}}
\newcommand{\e}[0]{{\mathrm{e}}}

\newcommand{\rs}[0]{r_o}
\newcommand{\SO}[1]{\mathrm{SO(}#1\mathrm{)}}

\long\def\aureply#1{\ifhmode\newline\fi\noindent{\bf Author Reply}:\ {\em #1}}

\def\gettimestamp#1<#2>{\def\timestamp{\tt #2}}

\usepackage[margin=0pt]{caption}
\gettimestamp Time-stamp: <2021-07-14 11:21:15>

\begin{document}

\title[KAM tori]{Invariant tori for multi-dimensional integrable hamiltonians coupled to a single thermostat}
\author{Leo T. Butler}
\address{Department of Mathematics, University of Manitoba, Winnipeg,
  MB, Canada, R2J 2N2}
\email{leo.butler@umanitoba.ca}
\date{\timestamp}

\begin{abstract}
  This paper demonstrates sufficient conditions for the existence of KAM
  tori in a singly thermostated, integrable hamiltonian system with $n$
  degrees of freedom with a focus on the generalized, variable-mass
  thermostats of order $2$--which include the Nos\'e thermostat, the
  logistic thermostat of Tapias, Bravetti and Sanders, and the Winkler
  thermostat. It extends Theorem 3.2 of Legoll, Luskin \& Moeckel, ({\em
    Non-ergodicity of {N}os\'e-{H}oover dynamics}, Nonlinearity, 22
  (2009), pp.~1673--1694) to prove that a `typical'' singly
  thermostated, integrable, real-analytic hamiltonian possesses a
  positive-measure set of invariant tori when the thermostat is weakly
  coupled. It also demonstrates a class of integrable hamiltonians,
  which, for a full-measure set of couplings, satisfies the same
  conclusion.
\end{abstract}

\maketitle

\section{Introduction} \label{sec:intro} A central model in
statistical mechanics is an isolated mechanical system, modeled by a
Hamiltonian $H$, that is in equilibrium with a heat bath at the
temperature $T$. Khinchin stressed ergodic theory as the foundations
of statistical mechanics--approximately twenty years after Fermi's
ill-fated effort to prove the quasi-ergodic hypothesis for mechanical
systems~\cite{MR0029808,springer_jour10.1007/BF02959600}. By the early
nineteen-sixties, the Kolmogorov-Arnol'd-Moser theory and the
Fermi-Pasta-Ulam numerical results demonstrated that there are
fundamental difficulties in the project to reduce statistical
mechanics to classical
mechanics~\cite{MR0068687,MR0097598,MR0163025,MR0170705,MR0147741};~\cite{MR2402700FPU,MR2402700}.

This note studies thermostated systems used in the computational
statistical-mechanics and molecular-dynamics literature. Many of the
results in the mathematical literature on the properties of these
systems are negative, in a sense: they demonstrate the existence of
positive-measure sets of KAM tori and hence the failure of ergodicity.

\subsection{Thermostated mechanics}
\label{sec:therm-mech}

Nos{é} \cite{nose,doi:10.1080/00268978400101201}, after Andersen
\cite{andersen}, introduced an ``extended system'' to create a system
with an invariant measure that projects onto the Gibbs-Boltzmann
measure in physical phase space. This consists of adding an extra
degree of freedom $s$, re-scaling momentum by $s$, and coupling the
extra state $s$ thus:
\begin{align}
  \label{eq:nose}
  𝐇 &= H(q,p s^{-1}) + N_T(s,p_s), & \text{where }N_T(s,p_s) =\dfrac{1}{2 Q} p_s^2 + g k T \ln s,
\end{align}
where $g$ is a parameter, $Q$ is the ``mass'' of the thermostat (a
proxy for the thermal coupling) and $k$ is Boltzmann's constant. In
this note, it is assumed that units are chosen so that
$gk=1$.\footnote{Nos{é} notes that $g=n+1$ or $g=n$ can be used: the
  former choice ensures that the micro-canonical ensemble of $𝐇$
  projects to that of $H$; the latter choice is appropriate when one
  views the momentum $p'=p/s$ to be ``real'' and one wishes to obtain
  ergodic averages in ``real''
  variables~(see~\cite[eq. 2.10]{doi:10.1080/00268978400101201},
  \cite[eq. 2.5, p. 513]{nose}).} Solutions to Hamilton's equations
for $𝐇$ model the evolution of the state of the infinitesimal system
along with the exchange of energy with the heat
bath~\cite[p. 187]{doi:10.1080/00268978600100141}.

Hoover's reduction eliminates the state variable $s$ and re-scales time
$t$~\cite{hoover}:
\begin{equation*}
  q = q,\qquad ρ = ps^{-1},\qquad \ddt{τ} = s \ddt{t},\qquad z = \ddt[s]{t}.
\end{equation*}
The Nos{é}-Hoover thermostat for an $n$-degree of freedom Hamiltonian
$H$ can be put in the form (c.f.~\cite[eq. 6]{hoover}):
\begin{align}
  \label{eq:nose-hoover}
  \dot{q} &= H_{ρ}, && \dot{ρ} = -H_q - ε \xi ρ, &&
  \dot{\xi} = ε \left( ρ · H_{ρ} - g k T \right),
\end{align}
where $ε ² = 1/Q$ and $z=ε ξ$.\footnote{Hoover puts $g=n$ to ensure
  that the extended system has an invariant density that projects to
  the Gibbs-Boltzmann density for $H$.}

Hoover observed that this thermostat is ineffective in producing the
statistics of the Gibbs-Boltzmann distribution from single orbits of
the thermostated harmonic oscillator~\cite{hoover}. There are numerous
extensions of the Nos{é}-Hoover thermostat that model the exchange of
energy with the heat bath using a single, additional thermostat
variable ($ξ$ in \ref{eq:nose-hoover}), the so-called \textem{single
  thermostats}. A sample
includes~\cite{proquest1859245445,PhysRevE.75.040102,doi:10.1142/S0218127416501704,10.12921/cmst.2016.0000061,Wang2015,MR3432739,MR3674309}. Winkler,
in~\cite{proquest1859245445}, introduces a $p/s^2$ coupling which
produces better results than Nos{é}'s, but precludes the Hooverian
reduction~\eqref{eq:nose-hoover} since it is equivalent to a
variable-mass variant the Nos{é} thermostat--see \cite[Lemma
3.1]{1909.02995}. In~\cite{PhysRevE.75.040102}, Watanabe \& Kobayashi
thermostat a harmonic oscillator with a thermostat that controls a
single moment of momenta and show its first-order averaged system is
integrable. Rech, in~\cite{doi:10.1142/S0218127416501704}, considers
the same variant of the Nos{é}-Hoover thermostated harmonic
oscillator--introduced in~\cite{SprottJulienClinton2014Hcat} by
Sprott, Hoover \& Hoover--in which the temperature varies with
position. Tapias, Bravetti \& Sanders,
in~\cite{10.12921/cmst.2016.0000061}, replace the linear friction of
the Nos{é}-Hoover thermostat with a $\tanh$-friction that saturates at
large magnitudes of the thermostat state $ξ$. Wang \& Yang,
in~\cite{Wang2015,MR3432739}, investigate the Nos{é}-Hoover
thermostated harmonic oscillator and find regions of phase space where
apparently chaotic dynamics exist and regions where invariant tori
appear. In~\cite{MR3674309}, the same authors visit the variant of the
Nos{é}-Hoover thermostated harmonic oscillator that thermostats total
energy, and demonstrate (numerically) the existence of a
horseshoe. The present author shows, in a $2$-degree of freedom
hamiltonian that is integrable and enjoys a saddle-centre critical
point, that the Nos{é}-Hoover thermostat splits the homoclinic
connections and creates horseshoes for a suitable parameter
regime~\cite{1909.02995}.

The Nos{é} and Nos{é}-Hoover thermostats are used to thermostat mixed
quantum-classical systems, too. Grilli \& Tosatti, in
\cite{PhysRevLett.62.2889} (see also~\cite{PhysRevLett.107.179902})
use a variant of Nos{é}'s thermostat to couple a combined quantum
and classical system with a heat
bath. In~\cite[eqs. 1--3]{MauriF1993CSAo}, Mauri, Car \& Tosatti
couple a mixed quantum-classical system to a Nos{é}-Hoover thermostat;
this work is extended by Alonso,
et.~al.~\cite{Alonso_10.1088/1751-8113/44/39/395004}. Sergi and Sergi
\& Petruccione, in~\cite{MR2321662,MR2426026}, examine an alternative
mixed quantum-classical system, based on Wigner's formalization,
coupled to a Nos{é}-Hoover thermostat (and chains). Thermostats are
applied to quantum systems by Mentrup \&
Schnack~\cite{MENTRUP2001337,MENTRUP2003370}.

\subsection{KAM Tori}
\label{sec:kam-tori}

Most notable for the purposes of the present note, there are several
studies of the Nos{é}-Hoover thermostat from the point-of-view of
near-integrable
systems~\cite{MR2299758,MR2519685,Mahdi20111348}. Legoll, Luskin \&
Moeckel, in~\cite{MR2299758}, show that the Nos{é}-Hoover thermostated
harmonic oscillator enjoys KAM tori near the $ε=0$ decoupled
limit;~they extend this result in~\cite{MR2519685} to show that an
integrable system that is coupled to a Nos{é}-Hoover thermostat has a
first-order averaged system that is also integrable near
$ε=0$~\cite[Theorem 3.2]{MR2519685}:

\begin{theorem}[Legoll, Luskin \& Moeckel]
  \label{thm:llm-ave}
  The averaged equations for the Nos{é}-Hoover
  dynamics~\eqref{eq:nose-hoover} for a completely integrable
  Hamiltonian system with $n$ degrees of freedom have $n$ independent
  first integrals.
\end{theorem}

The authors speculate that, even if the full system is ergodic,
finite-time orbit averages should converge very slowly to the spatial
average. They show numerically, using a rotationally-invariant, planar
mechanical hamiltonian with potential energy $V(r)=r^2+r^4$, that the
first integrals of the averaged thermostated system appear to display
no convergence to a spatial average.

\subsubsection{Rotationally-invariant potentials}
\label{sssec:llm-rot-inv}
The first result of the current note is the following:

\begin{theorem}
  \label{thm:llm-rot-inv}
  Let $M$ be a surface of revolution with local coordinates $(r,θ)$ on
  the rotationally-invariant open set $U ⊂ M$,
  $(r,θ) → \R^+ × \R/2 π \Z$. Let $V : M → \R$ be a real-analytic
  function which is rotationally invariant, so $V(q) = v(r)$. Assume
  that
  \begin{enumerate}
  \item[H1.] $v$ is strictly increasing on an interval $J$;
  \item[H2.] $r v''(r) + v'(r) > 0$ for all $r ∈ J$;
  \item[H3.] $T = \rs v'(\rs)$ for some $\rs ∈ J$;
  \item[H4.] The mechanical hamiltonian $H : \cotangent M → \R$ is
    defined by
    $$ H(r,p_r,θ,p_{θ}) = ½ \left( (c(r) p_r)^2 + (p_{θ}/r)^2 \right) + v(r), $$
    where the kinetic energy is induced by the second-fundamental form
    of $M$ (i.e. the natural metric induced by the inclusion
    $M ⊂ \R^3$).
  \end{enumerate}
  Then, the Nos{é}-thermostated hamiltonian $𝐇$~\eqref{eq:nose} for
  each such $T$ enjoys a full-measure set of masses $𝔔 ⊂ \R^+$ such
  that if $Q ∈ 𝔔$, then $𝐇$ enjoys a set of positive measure of
  invariant KAM tori.
\end{theorem}

This theorem is striking because, in contrast to previous work in the
area, there is no constraint that the thermostat mass be
``sufficiently small''. When this theorem is applied to the case
examined in~\cite{MR2519685} or a Lennard-Jones potential or the
spherical pendulum potential, one obtains:

%% graph of T(r) (red), v(r) (blue) and rv'' + v' (black)
%% yellow region is where H1+H2 are true
%% draw2d(x_voxel=50, y_voxel=50, fill_color=yellow, region(y>=0 and r>=(m/l)^(1/(m-l)) and r<=(m/l)^(2/(m-l)),r,1,1.5,y,-1,2), line_width=2, explicit(w, r, 0.9, 1.5), color=black, explicit(d2w*r+dw, r, 0.9, 1.5), yaxis_secondary=true, ytics_secondary=true, color=red, explicit(dw*r, r, 0.9, 1.5), grid=true, terminal=[wxt,1], yrange=[-1,2], yrange_secondary=[-1,1]), σ=1, m=12, l=6;
\begin{corollary}
  \label{cor:llm-rot-inv}
  If
  \begin{enumerate}
  \item $M=\R^2$, $v(r) = r^2+r^4$ and $J = \R^+$; or
  \item $M=\R^2$, $v(r) = r^{-12} - r^{-6}$ and $J = (0,3/4)$; or
  \item $M=\sphere{2}$, $r = \sin(φ)$, $v = -\cos(φ)$ for $0<φ<π$ and $J = \R^+$,
  \end{enumerate}
  then, for all $T ∈ J$, there is a full-measure set of masses
  $𝔔 ⊂ \R^+$ such that if $Q ∈ 𝔔$, then $𝐇$ enjoys a positive-measure
  set of invariant KAM tori.
\end{corollary}

In fact, Theorem~\ref{thm:llm-rot-inv} and its corollary are proven
for a more general class of single thermostat, called
\defn{generalized, variable-mass thermostats of order $2$} (see
definition~\ref{def:order-2-thermostat-generalized}), which include
the thermostats of Winkler and Tapias, Bravetti \&
Sanders~\cite{proquest1859245445,10.12921/cmst.2016.0000061}.

I believe that for a ``generic'' rotationally-invariant potential $v$,
and temperature $T$, the set of masses $𝔔$ is $\R^+$ less a finite
set. However, the calculations used to prove
Theorem~\ref{thm:llm-rot-inv} do not appear to lend themselves to a
proof of this.

\subsubsection{Thermostating integrable hamiltonians}
\label{sssec:llm-int}

The current paper sharpens Theorem~\ref{thm:llm-ave} of Legoll, Luskin
\& Moeckel and shows that their speculation is true for ``typical''
real-analytically integrable hamiltonians. To explain, some
terminology and definitions are needed. In the sequel, $M$ is a
real-analytic $n$-dimensional manifold; $\cotangent M$ is its
cotangent bundle; $H : \cotangent M → \R$ is a real-analytic function
(a hamiltonian) that is completely integrable with respect to the
canonical Poisson bracket (see \S \ref{sec:int-ham} below);
$U ⊂ \cotangent M$ is a toroidal cylinder (i.e. a symplectomorph of
$\T^n × B$ where $B ⊂ \R^n$ is a diffeomorph of the unit ball);
$(θ,I) : U → \T^n × B$ are real-analytic angle-action variables for
$H$; $H | U = G \circ I$ for some real-analytic function $G$. The
\defn{instantaneous temperature} of the system $H$ at the point
$(q,p) ∈ \cotangent M$ is $\kp{\dot{q}}{p}$ and the \defn{orbit mean
  temperature}, $κ$, is the Birkhoff average of the instantaneous
temperature. It is shown that $κ(I) = \kp{\D{G}(I)}{I}$ (see \S
\ref{sec:temp-fun}). Finally, the following notion of torsion is
needed.

\begin{definition}[\cite{MR2505319}]
  \label{def:r-non-deg}
  A real-analytic map $f : U → \R^m$ is R-non-degenerate if the
  smallest linear subspace that contains $f(U)$ is $\R^m$.
\end{definition}

This geometric definition of R-non-degeneracy is equivalent to the
more standard definition that the partial derivatives of $f$ at a
point span $\R^m$.

The main result of the current paper is:

\begin{theorem}
  \label{thm:main-thm-1}
  Let $H : \cotangent M → \R$ be a real-analytic hamiltonian that is
  completely integrable with real-analytic integrals. Assume that
  $T>0$ is a regular value of the orbit mean temperature function $κ$
  and let
  $$\thermostateqm{T} = \set{(θ,I,s,p_s) \mid κ(I/s)=T, p_s=0}.$$ Then,
  there exists a real-analytic symplectomorphism $φ$,
  $(θ,I,s,p_s)=φ(\hat{θ},\hI,v,V)$, whose image contains a
  neighbourhood of $\thermostateqm{T}$, and a real-analytic function
  $\bar{G}$ such that the Nos{é}-thermostated hamiltonian $𝐇$ is
  transformed to
  \begin{equation}
    \label{eq:main-thm-normal-form}
    𝐇 ⎄ φ = \bar{G}(\hat{I}) + ½ ε α(\hI) \left( v^2 + V^2 \right) + ε^{\frac{3}{2}} 𝐏_{ε} ⎄ φ
  \end{equation}
  where $ε^2 = 1/Q$, $𝐏_{ε}$ is real analytic in $ε$ for $ε>0$,
  continuous in $ε$ at $ε=0$ and real analytic in the other
  variables.

  If the re-scaled frequency map $Ω=(\D{\bar{G}},α)$ is
  R-non-degenerate, then, there exists a $Q_o = Q_o(T,G) > 0$ such
  that for each $Q ∈ (Q_o,∞)$ there is a neighbourhood of
  $\thermostateqm{T}$, such that the hamiltonian $𝐇$ has a
  positive-measure set of invariant tori in that neighbourhood.
\end{theorem}

Lemma~\ref{lem:normal-form-near-T} has more information about the
function $\bar{G}$ and the neighbourhood of the thermostatic
equilibrium set $\thermostateqm{T}$. Because the first step in
creating the symplectomorphism $φ$ is indirect, an explicit
construction of $\bar{G}$ seems impossible (except in some special
cases, see remarks~\ref{rem:normal-form-near-T-example}
and~\ref{rem:normal-form-near-T-example-contd} below). In lieu of this
construction, one is able to prove:

\begin{corollary}
  \label{cor:main-cor-1}
  Assume the hypotheses of Theorem~\ref{thm:main-thm-1}. Then, the
  re-scaled frequency map $Ω$ of $𝐇$ can be made R-non-degenerate by
  means of a $C^2$-small perturbation of $H$ in the space of
  real-analytically integrable hamiltonians.
\end{corollary}

See Corollary~\ref{cor:main-cor} for the precise, somewhat technical,
statement of this corollary.

\begin{remark}
  \label{rem:main-remark}
  Like Theorem~\ref{thm:llm-rot-inv}, Theorem~\ref{thm:main-thm-1} and
  its corollary hold for generalized, variable-mass thermostats of
  order $2$. In addition, the theorems hold in the smooth ($C^{∞}$)
  category due to the work of Herman-F{é}joz~\cite{MR2104595}. A
  puzzling aspect of Theorem~\ref{thm:main-thm-1} is that the
  Nos{é}-thermostated harmonic oscillator produces a normal form
  $𝐇 ⎄ φ$ whose re-scaled frequency map is \textem{not}
  R-non-degenerate (see
  remark~\ref{rem:normal-form-near-T-example-contd}). So, the theorem
  here is not proven by reworking an existing proof, but contributes
  some novel ideas. In addition, it is likely that a better expansion
  of $𝐇 ⎄ φ$ in the thermostat variables $(v,V)$ will also give a
  proof that the Nos{é}-thermostated 1-d harmonic oscillator has a
  frequency map that is R-non-degenerate.
\end{remark}

\subsection{Outline}
\label{sec:outline}

This note is organized as follows: \S 2 reviews salient facts about
integrable hamiltonian systems and the effect of momentum re-scalings
on such systems; \S 3 defines a general class of single thermostats
for which the stated theorems can be proven; \S 4 contains the
material to prove Theorem~\ref{thm:main-thm-1} and shows how the
theorem implies an improvement on Nos{é}'s heuristic approximation of
the frequency of the thermostat's oscillations; \S 5 proves
Theorem~\ref{thm:llm-rot-inv} and shows numerical calculations that
confirm the theorem's predictions; \S 6 contains the necessary
material from properly degenerate KAM theory.

\section{Integrable Hamiltonians}
\label{sec:int-ham}

Let $M$ be an $m$-dimensional, real-analytic manifold. The type of
Hamiltonian that is considered in this paper is a real-analytic
function $H : \cotangent M → \R$ defined on the cotangent bundle
$\cotangent M = \set{(q,p) \mid q ∈ M, p ∈ \cotangent_q M }$ of the
real-analytic manifold $M$. Say that $F : \cotangent M → \R$ enjoys
\defn{fibre-wise super-linear growth} if, for each, $(q,p) ≠ (q,0)$,
$F(q,σ p)/σ → ∞$ as $σ → ∞$. The point-wise (or instantaneous)
temperature at $(q,p)$ is defined to be $\kp{p}{H_p}$. If the
instantaneous temperature function is bounded above by a constant $C$,
then a simple comparison shows that $|H(q,σ p)/σ|$ is bounded above by
$|H(q,p)|/σ + C \ln(σ)/σ$, so $H$ cannot enjoy super-linear
growth. Henceforth, it is assumed that $H$ enjoys super-linear growth,
whence the instantaneous temperature function is unbounded above on
$\cotangent_q M$ for each $q ∈ M$.

The cotangent bundle $\cotangent M$ carries the tautological Liouville
1-form $λ$, canonical symplectic form $Ω = \D{λ}$ and its dual, the
Poisson bracket $\pb{}{}$. If $q=(q_1,\ldots,q_m)$ are local
coordinates on $M$, then any $1$-form in $\cotangent_q M$ is uniquely
expressed as $p_1 \D{q}_1 + \cdots + p_m \D{q}_m$ for some scalars
$p_1, \ldots, p_m$. These ``adapted'' coordinates
$(q_1,\ldots,q_m,p_1,\ldots,p_m)$ on $\cotangent M$ satisfy the
properties
\begin{align} \label{al:poisson-bracket}
  λ&=\sum_{i=1}^m p_i\, \D{q}_i, &&& \pb{q_i}{p_j} &= δ_{ij}, &&& \pb{q_i}{q_j}&=\pb{p_i}{p_j}=0,
\end{align}
where $δ_{ij}$ is Kronecker's delta-function. A \defn{Darboux system
  of coordinates} is a coordinate system
$(x_1,\ldots,x_m,y_1,\ldots,y_m)$ on $\cotangent M$ that satisfies the
above Poisson bracket relations. When
$(x_1,\ldots,x_m,y_1,\ldots,y_m)$ is a Darboux system of coordinates,
the 1-form $\sum_{i=1}^m y_i \D{x}_i$ equals the Liouville 1-form $λ$
up to the addition of a closed 1-form.

The Poisson bracket endows the space of real-analytic (resp. smooth)
functions on $\cotangent M$ with the structure of a Lie algebra.

Recall the notion of complete integrability:

\begin{definition} \label{def:complete-integrability}
  Let there exist a real-analytic map $F : \cotangent M → \R^m$, where
  $m$ is the dimension of $M$, such that
  \begin{enumerate}
  \item the components of $F$ Poisson commute;
  \item $F$ has a regular value;
  \item $H = h \circ F$ is the pull-back of a real-analytic $h : \R^m
    → \R$.
  \end{enumerate}
  Then, $H$ is said to be \defn{completely integrable} and $F$ is a
  first-integral map for $H$.
\end{definition}

The significance of complete integrability is due to the following

\begin{theorem}[Liouville--Arnol'd--Duistermaat]
  \label{thm:louville-arnold}
  Let $H : \cotangent M → \R$ be completely integrable with
  first-integral map $F : \cotangent M → \R^m$. The set of regular
  points of $F$, $L ⊂ \cotangent M$, is an open and dense set that is
  fibred by Lagrangian tori and satisfies
  \[
    \xymatrix@R10mm@C10mm@M3mm{
                         & \T \ar@{^{(}->}[r]^{{ι}} \ar@{_{(}->}[d]^{{ι}} & L \ar@{.>}[d]_{π} \ar@{^{(}->}[r]^{{ι}}   & \cotangent M \ar@{->}[dr]^F \\
       P \ar@{^{(}->}[r] & \cotangent B \ar@{->>}[r]                      & {\mathfrak L}=\cotangent B/P \ar@{->>}[r] & B \ar@{.>}[r] & \R^m        \\
      }
  \]
  where $\xymatrix{\ar@{.>}[r]&}$ denotes a local isomorphism, $P$ is
  a sub-bundle $\cotangent B$ such that for each $b ∈ B$, $P_b$ is the
  set of pullbacks $π^*(\D{f})_b$ of locally periodic hamiltonians
  $π^* f$. The quotient $\cotangent_b B / P_b$ is a lagrangian torus
  so that $L$ is locally isomorphic to ${\mathfrak L}$. The
  obstruction to a global isomorphism is a $2$-dimensional Chern class
  that was identified by Duistermaat~\cite{MR596430,MR906389}.
\end{theorem}

In the classical Liouville-Arnold theorem, the focus is on a
neighbourhood of a single, regular Lagrangian torus $L_b = π^{-1}(b)$
for some $b ∈ B$. In that setting, there is a neighbourhood $U ⊂ B$ of
$b$ with coordinates $(I_1,\ldots,I_n)$ where the functions $π^*(I_j)$
generate periodic hamiltonian flows on $π^{-1}(U) ⊂ L$ with unit
primitive period. The tautological Liouville $1$-form
$λ = \sum_{i=1}^n θ_i\,\D{I}_i$ is well-defined on $\cotangent_U B$
and defines a local diffeomorphism $\cotangent_U B → π^{-1}(U)$ that
induces a diffeomorphism $\cotangent_U B/P_U → π^{-1}(U)$ when the
$θ_i$ are taken to be defined $\bmod 1$. This implies that the
pull-back of the Poisson bracket $\pb{}{}_U$ on $π^{-1}(U)$ to
$\cotangent_U B/P_U$ satisfies
\begin{align} \label{al:poisson-bracket-U}
  \pb{I_i}{θ_j}_U &= Δ_{ij}, &&& \pb{I_i}{I_j}_U &= \pb{θ_i}{θ_j}_U = 0,
\end{align}
where $Δ_{ij}$ is a constant, non-singular matrix. In fact, $Δ$ is an
integer matrix with an integral inverse since the flows of the
hamiltonian vector fields $\pb{}{I_j}$ and $\pb{}{I_j}_U$ are periodic
with unit primitive period. This implies that there are angle variable
$φ_i$ defined implicitly by $θ_i = \sum_{j=1}^m Δ_{ij} φ_j$ for
$i=1,\ldots,m$ and that satisfy $\pb{I_i}{φ_j}_U=δ_{ij}$.

In the sequel, the neighbourhood $π^{-1}(U)$ will be suppressed in
discussions involving action-angle variables.

\subsection{Re-scalings}
\label{sec:rescaling}

Let $σ ≠ 0$ be a non-zero constant. The diffeomorphism
$φ_{σ} : \cotangent M → \cotangent M$ is defined by re-scaling $p$:
$φ_{σ}(q,p) = (q, σ p)$. Let $f_{σ} = f \circ φ_{σ}$ be the
composition of the map $f$ with $φ_{σ}$.

\begin{proposition}
  \label{prop:rescaling}
  If $H : \cotangent M → \R$ is completely integrable with
  first-integral map $F : \cotangent M → \R^m$, then $H_{σ}$ is
  completely integrable with first-integral map $F_{σ}$.

  In addition, if $(I_i, θ_j)$ are action-angle variables for $H$,
  then $(σ^{-1} I_{σ,i}, θ_{σ,j})$ are action-angle variables for
  $H_{σ}$.
\end{proposition}

\begin{proof}
  An easy computation shows that if $f, g$ are analytic functions on
  $\cotangent M$ then
  \begin{equation}
    \label{eq:rescaling-pb}
    \pb{f_{σ}}{g_{σ}} = σ \pb{f}{g} \circ φ_{σ}.
  \end{equation}
  This, coupled with \eqref{al:poisson-bracket-U}, proves the
  proposition.
\end{proof}

Let $I = (I_1,\ldots,I_m)$ be a local system of actions for $H$ and
let $J_{σ} = σ^{-1} \, (I_{σ,1},\ldots,I_{σ,m})$ be the actions for
$H_{σ}$. Because $H$ is completely integrable, there is an analytic
function $G$ such that $H = G \circ I$. Then $H_{σ} = G \circ I_{σ}$
and therefore
\begin{equation}
  \label{eq:hlambda}
  H_{σ} = G(σ \, J_{σ}),
\end{equation}
where both $H_{σ}$ and $J_{σ}$ are functions of $(q,p)$.

\begin{proposition}
  \label{prop:hlam-expansion}
  Let $R(θ,I,σ) = H_{σ} - G(σ I)$. Then
  \begin{equation}
    \label{eq:hlam-expansion}
    R = O(σ-1).
  \end{equation}
\end{proposition}
\begin{proof}
  By the Taylor expansion about $σ=1$, $J_{σ} = J_1 + O(σ-1)$, and
  $J_1=I$. Therefore, $I_{σ} = σ J_{σ} = σ I+ O(σ-1)$. The proposition
  follows from \eqref{eq:hlambda}.
\end{proof}

The hamiltonian function $G(σ I)$ is a truncation of $H_{σ}$ that is
simple enough to analyze on a fixed system of action-angle
coordinates. Unfortunately, it is not a good enough integrable
truncation of $H_{σ}$ to be exploited when $H$ is coupled to a single
thermostat via momentum re-scaling.

\begin{proposition}
  \label{prop:hlam-gen-fun}
  Assume the hypotheses of Proposition~\ref{prop:rescaling}. Then, for
  each toroidal cylinder $V ⊂ L$ with angle-action coordinates
  $(θ,I) : V → \T^n × B$, and $σ$ sufficiently close to $1$, there is
  a generating function $Σ = Σ(θ,J_{σ};σ)$ of a symplectic map
  $f_{σ} : V → φ_{σ^{-1}}(V)$ and a real-analytic function
  $\hat{G} = \hat{G}(I;σ)$ such that
  \begin{align}
    \label{eq:hlam-gen-fun}
    (θ_{σ},J_{σ}) &= f_{σ}(θ,I) &&& \textrm{and\ } \\
    \label{eq:hlam-ghat}
    H_{σ} ⎄ f_{σ}(θ,I) &= \hat{G}(I;σ).
  \end{align}
\end{proposition}

\begin{proof}
  Define $f_{σ}$ by \eqref{eq:hlam-gen-fun}, where $σ$ is close enough
  to $1$ so that $φ_{σ^{-1}}(V) ⊂ L$. The map $f_{σ}$ is symplectic
  because it is a canonical change of coordinates from one set of
  angle-action variables $(θ,I)$ to another $(θ_{σ},J_{σ})$. Moreover,
  such changes of angle-action variables are induced by a generating
  function, hence the existence of $Σ$. Finally, by
  Proposition~\ref{prop:rescaling} and \eqref{eq:hlambda}, one obtains
  that the pullback of $H_{σ}$ by $f_{σ}$ is a real-analytic function
  of $I$ that is parameterized by $σ$. Hence~\eqref{eq:hlam-ghat}
  holds.
\end{proof}

\begin{remark}
  \label{rem:other-rescalings}
  The diffeomorphism $φ_{σ}$ effects a momentum re-scaling, which is
  used in single thermostats following Nos{é}'s lead. On the other
  hand, Andersen's barostat re-scales both momentum $p$ and spatial
  coordinates $q$, as does the classical version of Grilli \&
  Tosatti's thermostat. This is perfectly consistent with the
  framework introduced here for momentum re-scaling, provided that the
  configuration manifold $M$ is invariant under such a re-scaling of
  spatial coordinates (e.g. $M$ is the complement of a finite number
  of hyperplanes in $\R^m$). Since the spatial and momentum re-scalings
  commute and \eqref{eq:rescaling-pb} holds for each separately, it
  holds for the composition, too. That is, if $α,β$ are non-zero
  scalars and $σ=(α,β)$, $n(σ)=α β$, then define the re-scaling
  diffeomorphism $φ_{σ}(q,p) = (α q, β p)$ of $\cotangent M$. Then,
  \eqref{eq:rescaling-pb} becomes
  \begin{equation}
    \label{eq:rescaling-pbab}
    \pb{f_{σ}}{g_{σ}} = n(σ) \pb{f}{g} \circ φ_{σ}
  \end{equation}
  where $f_{σ}=f \circ φ_{σ}$, etc.. Propositions
  \ref{prop:rescaling}---\ref{prop:hlam-gen-fun} hold with the
  appropriate replacement of the scalar $σ$ by $σ=(α,β)$.
\end{remark}

\section{Single Thermostats}
\label{sec:single-thermostats}

In~\cite{1806.10198v3}, the following definition is introduced.

\begin{definition}
  \label{def:order-2-thermostat}
  A $C^r$, $r>2$, function $N_T(s,S) = Ω(s) F(S) + T \ln s$ is a
  \defn{variable-mass thermostat of order $2$} if $Ω>0$,
  $F(0)=F'(0)=0$, $F''(0)>0$ and $F'$ vanishes only at $0$.
\end{definition}

In the present paper, we will always assume that $r=ω$, i.e. the
thermostat is real analytic. As noted in~\cite{1806.10198v3}, {\em
  order} means the order of the first non-trivial term in the
Maclaurin expansion of $F$ and the thermostat {\em mass} is
$1/Ω(s)$. When $Ω$ is constant, the thermostat is called
\defn{elementary}. Higher-order thermostats are in the literature, but
they present serious challenges for the techniques of this paper. The
following are examples of real-analytic thermostats of order $2$:

\begin{enumerate}
\item the Nos{é}-Hoover thermostat~\cite{doi:10.1080/00268978400101201,nose,hoover};
\item Tapias, Bravetti \& Sanders logistic
  thermostat~\cite{10.12921/cmst.2016.0000061,10.12921/cmst.2017.0000005};
\item\label{it:winkler} Winkler's thermostat and its generalization~\cite{proquest1859245445}.
\end{enumerate}

A note about Winkler's thermostat: it appears in the form of
\eqref{eq:nose} but with $p s^{-1}$ replaced by $p s^{-2}$ and more
generally one may use the coupling $p s^{-\e}$ for any $\e > 0$. By a
change of variables, this is equivalent to a thermostat with the
standard coupling $p s^{-1}$ and controller
$N_T(s,S) = \frac{a}{2} (\e s^{1-1/\e} S)^2 + T \ln s$. Such a
thermostat is called a \defn{generalized Winkler thermostat}.

One can extend the definition of a variable-mass thermostat of order
$2$ to include real-analytic hamiltonians of the form
\begin{equation}
  \label{eq:order-2-thermostat-generalized}
  N_T(s,S) = T \ln s + \underbrace{\sum\limits_{k=2}^{∞} \dfrac{1}{k!} Ω_k(s) S^k }_{F(s,S)},
\end{equation}
where each $Ω_k$ is real-analytic on $\R^+$, $F$ is non-negative,
$F_S = 0$ only along $\R^+ × \set{0}$ and $F_{SS} > 0$ along
$\R^+ × \set{0}$. These assumptions imply, in particular, that
$Ω_2(s) > 0$ for all $s$.

\begin{definition}
  \label{def:order-2-thermostat-generalized}
  A real-analytic function $N_T$ satisfying the properties in the previous paragraph
  is called a \defn{generalized variable-mass thermostat of order $2$}.
\end{definition}

One might be curious as to why the form of the potential is always
$T \ln s$ in the above definitions. In the non-hamiltonian setting,
Jellinek, Jellinek \& Berry and
Ramshaw~\cite{JellinekJulius1988Dfns,scopus2-s2.0-0001730129,scopus2-s2.0-4243944612,PhysRevE.92.052138}
look at a general form for Nos{é}-Hoover-like thermostats and none of
these authors make an equivalent assumption. However, those papers do
not attempt to derive a hamiltonian form for their non-hamiltonian
thermostats and do not address the problem of a canonical form, which
is a non-trivial problem since they are parameterized by several
functional parameters. On the other hand, under fairly modest
hypotheses--the thermostat acts via momentum re-scaling and the
thermostatic equilibrium set projects to
$\set{(s,S) \mid S=0, s>0}$--then it is known that, up to a symplectic
change of variables in $(s,S)$, the potential can be taken to be
$T \ln s$~(\cite[section 5.1]{nlin2017}, see also~\cite[part
C]{nose}). The cost, though, is that the function $F$ takes a quite
general form as in~\eqref{eq:order-2-thermostat-generalized}.

\section{Singly-Thermostated, Integrable Hamiltonians}
\label{sec:sing-therm-int-ham}

Let $N_T$ be a real-analytic, generalized variable-mass thermostat of
order $2$ and $H$ be a real-analytically completely integrable
hamiltonian.

\subsection{An integrable truncation}
\label{sec:int-trunc}

An integrable truncation of $𝐇$, denoted by $\mean{𝐇}_0$ and the remainder
$𝐑 = 𝐇 - \mean{𝐇}_0$, are given by
\begin{align}
  \label{eq:int-trunc-Hb}
  \mean{𝐇}_0(I,s,S) &= G(I/s) + N_T(s,S), &&& 𝐑(θ,I,s,S) &= R(1-1/s,θ,I)
\end{align}
where $σ = 1/s$ in \eqref{eq:hlambda} and $R$ is defined in
Proposition~\ref{prop:hlam-expansion}.

\begin{proposition}
  \label{prop:int-trunc}
  The hamiltonian function $\mean{𝐇}_0$ is completely integrable.
\end{proposition}

The proof follows by virtue of the fact that the $n$ components of $I$
and $\mean{𝐇}_0$ are functionally independent, Poisson-commuting first
integrals of the $n+1$-degree of freedom hamiltonian $\mean{𝐇}_0$. The
proposition justifies the description of $\mean{𝐇}_0$ as an
\defn{integrable truncation} of $𝐇$. In the ideal case, it would
suffice to use $\mean{𝐇}_0$ to study the normal form of $𝐇$ in a
neighbourhood of a $\mean{𝐇}_0$-invariant torus. As it turns out, this
does work when $G$ is homogeneous in $I$, but otherwise, one needs a
finer truncation.

\subsection{The temperature function and its Birkhoff average}
\label{sec:temp-fun}

Let us define a phenomenological temperature function for the
hamiltonian $H$.

\begin{definition}
  \label{def:temph}
  Let $H : \cotangent M → \R$ be a real-analytic hamiltonian. For each
  $(q,p) ∈ \cotangent M$, and non-negative integer $n$, define
  \begin{equation}
    \label{eq:temph}
    \tempH{n}{H}(q,p) = \left.\didin{n}{\phantom{σ}}{σ}\right|_{σ=1} H(q,σ p).
  \end{equation}
  The function $\tempH{1}{H}(q,p)$, which equals $\kp{p}{H_2(q,p)}$,
  defined to be the \defn{instantaneous temperature} of the system
  with state $x=(q,p)$.
\end{definition}

\begin{remark}
  \label{rem:temph}
  The functions $\tempH{n}{H}$, for $n ≥ 1$, define a type of moment
  of momentum, while the ratios $\tempH{n+2}{H}/\tempH{n}{H}$, define
  a type of ``weighted temperature''. In $1$-degree of freedom, the
  moments are used in the thermostats of Watanabe \&
  Kobayashi~\cite{PhysRevE.75.040102} and in thermostats that target
  higher moments of
  temperature~\cite{sciversesciencedirect_elsevier0375-9601.95.00973-6,10.1016/j.cnsns.2015.08.02,doi:10.1016/j.physleta.2015.08.034}.

  The function $\tempH{n}{H}(q,p)$ is related to the $n$-th fibre
  derivative of $H$, $\nthderivative{n}{H}(q,p)$, via the identity
  \begin{equation}
    \label{eq:temph-nthder}
    \tempH{n}{H}(q,p) = \kp{ p^{(n)} }{ \nthderivative{n}{H}(q,p) }
  \end{equation}
  where
  $p^{(n)} = \underbrace{p \otimes \cdots \otimes p}_{n\
    \mathrm{times}}$ is the $n$-fold symmetric tensor product of $p$
  with itself.
\end{remark}

Let $φ^t$ be the complete flow of the hamiltonian $H$ and let $f :
\cotangent M → \R$ be a continuous function. The Birkhoff
average of $f$, denoted by $\mean{f} $, is
\begin{equation}
  \label{eq:mean-f}
   \mean{f} (q,p) = \lim_{T → ∞} \dfrac{1}{2T} \int_{-T}^T f(φ^t(q,p))\, \D{t}.
\end{equation}

\begin{definition}
  \label{def:mean-temph}
  The \defn{orbit mean temperature} is the function $κ(q,p)$ that is
  the Birkhoff average of $\tempH{1}{H}$; more generally, the
  \defn{orbit mean moment of momentum} is the function $κ_n(q,p)$ that
  is the Birkhoff average of $\tempH{n}{H}$ for $n ≥ 1$.
\end{definition}

\begin{proposition}
  \label{prop:mean-temph}
  Let $H$ be completely integrable. If $(q,p) ∈ L$ (the union of
  regular Liouville tori), then
  \begin{equation}
    \label{eq:mean-temph}
    κ = \kp{I}{\D{G}(I)} = \left. \didi[{\phantom{σ}}]{σ} \right|_{σ=1} G(σ I),
  \end{equation}
  where $H = G \circ I$.
\end{proposition}
\begin{proof}
  Let $γ(t) = (q(t),p(t)) = φ^t(q,p)$ be the integral curve through
  the point $(q,p) ∈ L$. Then, the closure of $γ(\R)$ is contained in
  a Lagrangian torus and so there are angle-action coordinates $(θ,I)$
  defined in a neighbourhood $U$ of the orbit $γ(\R)$, as described in
  the Liouville--Arnol'd--Duistermaat
  theorem~\ref{thm:louville-arnold}. Then, since
  $\sum_{i=1}^n \D{I}_i ∧ \D{θ}_i = \sum_{i=1}^n \D{p}_i ∧ \D{θ}_i$,
  there is an analytic function $S : U → \R$ such that $\D{S} +
  \sum_{i=1}^n I_i \D{θ}_i = \sum_{i=1}^n p_i \D{q}_i$.

  If we define $γ_T = γ([-T,T])$ to be a centred orbit segment and
  $κ_T$ to be the mean instantaneous temperature over this segment,
  then
  \begin{align}
    \label{eq:temph-bk}
    2 T κ_T(q,p) &= \int_{-T}^T \kp{p(t)}{H_2(q(t),p(t))}\, \D t = \int_{-T}^T \kp{p(t)}{\dot{q}(t)}\, \D t, \\\notag
                 &= \int_{γ_T} p · \D{q} = \int_{γ_T} I · \D{θ} + S(γ(T)) - S(γ(-T)),\\\notag
                 &= \int_{-T}^T \kp{ω(I)}{I}\, \D{t} + S(γ(T)) - S(γ(-T)), &&& \mathrm{since\ } \dot{θ} = ω(I), \\\notag
                 &= 2T \kp{ω(I)}{I}  + S(γ(T)) - S(γ(-T)),
  \end{align}
  where $ω(I) = \D{G}(I)$. Since $S$ is continuous on $U$, it is
  bounded on a neighbourhood of $γ(\R)$, and $κ = \lim_{T → ∞} κ_T$,
  so the result follows.
\end{proof}

\begin{remark}
  \label{rem:invariance-of-orbit-mean-temp}
  Let $U,V ⊂ \cotangent M$ be open sets; a map $f : U → V$ is
  \defn{exact symplectic} if the pull-back of the tautological
  Liouville 1-form $λ | V$ under $f$ equals $λ | U$ up to the
  derivative of a scalar function. That is, $f$ is exact symplectic if
  $f^*(λ|V) = (λ + \D{S})|U$ for some $S : U → \R$. The preceding
  proof makes use of a single property of the transformation to
  action-angle coordinates: it is \defn{exact symplectic}. Indeed, the
  following is true:

  {\it If $f : U → V$ is exact symplectic, and $κ=κ_{H}$ is the
    orbit mean temperature function of the Hamiltonian $H$ on $V$, then the
    orbit mean temperature function $κ_{H ⎄ f}$ of the Hamiltonian $H
    ⎄ f$ on $U$ equals $κ ⎄ f$.}

  In short, the orbit mean temperature function is an invariant of
  exact symplectic changes of coordinates.
\end{remark}

\begin{remark}
  \label{rem:orbit-mean-space-mean}
  The Birkhoff average is used because it provides a very natural proof
  of proposition~\ref{prop:mean-temph}. On the other hand, the
  regularity of the function $\mean{f} $ is generally quite low (just
  $L^1$) even when $f$ is real analytic. The low regularity stems from
  the denseness of the ``resonant tori'' where the frequency vector
  $ω(I)$ enjoys non-trivial integral linear relations:
  $\exists k ∈ \Z-\set{0}$ such that $\kp{k}{ω(I)} = 0$. These
  resonances mean the functions $\exp(i \kp{k}{θ})$ do not average to
  $0$. On the other hand, the set of ``non-resonant tori'' is of full
  measure, and on each such torus, the Birkhoff average converges to
  the mean value over the torus.

  It follows that the Birkhoff average $\mean{f} $, of a real-analytic
  function $f$, is an equivalence class of integrable functions that
  contains a real-analytic representative, namely the fibre average of
  $f$. This latter is the average of $f$ on each Lagrangian torus, the
  average being computed with respect to the unique Haar measure on
  that torus.

  In the sequel, we prefer to identify $\mean{f}$ with this
  real-analytic fibre average.
\end{remark}

\subsection{Thermostatic equilibria and iso-thermal non-degeneracy}
\label{sec:therm-eq}

\begin{definition}
  \label{def:thermo-eq}
  Fix $T>0$ and let $κ$ be the orbit mean temperature
  function~\defref{def:mean-temph}. The set
  \begin{equation}
    \label{eq:thermo-eq}
    \thermostateqm{T} = \set{ (θ,I,s,S) \mid κ(I/s)=T,\, s>0,\, S=0 }
  \end{equation}
  is the set of thermostatic equilibria at temperature $T$ for the
  integrable truncated hamiltonian $\mean{𝐇}_0$~\eqref{eq:int-trunc-Hb}.
\end{definition}

\begin{proposition}
  \label{prop:thermo-eq}
  Assume that $T$ is a regular value of $κ$. Then, the thermostatic
  equilibrium set $\thermostateqm{T}$ is an invariant, real-analytic
  submanifold for the hamiltonian flow of $\mean{𝐇}_0$.
\end{proposition}

The proof of this proposition is a straightforward application of the
implicit function theorem to prove that there are an open neighbourhoods $W$
of $κ^{-1}(T)$ and $1 ∈ J$ such that $(I,s) ∈ W × J$ and $κ(I/s) = T$
iff $s=s_o(I)$.

\begin{remark}
  \label{rem:thermo-eq}
  The closure of $\thermostateqm{T} ⊂ \cotangent M$ is, in general, an
  intractable object. However, in case the Liouville foliation of the
  integrable hamiltonian $H$ is well-behaved (such as when the first
  integral map $F$ is
  \defn{non-degenerate}~\cite{MR1036125,MR1104922,MR1047477}), then
  $κ$ extends continuously to the closure of $\thermostateqm{T}$. In
  this situation, it is useful to regard the
  $\closure{\thermostateqm{T}}$ as the set of thermostatic
  equilibria. The utility of this point of view is highlighted by
  section~\ref{sec:llm-ex} where a periodic orbit $Γ$ in
  $\closure{\thermostateqm{T}} ∖ \thermostateqm{T}$ is used to compute
  the normal form of $𝐇$ in a neighbourhood of $Γ$ in order to apply
  KAM theory.
\end{remark}

\subsection{A normal form}
\label{sec:normal-form}

Let $\cotangent(\T^n × \R^+)$ be the cotangent bundle of $\T^n × \R^+$
and let $X_0 ⊂ \cotangent(\T^n × \R^+)$ be the zero section of that
bundle and, for $0<c<1$, let $X_c ⊂ X_0$ be the privileged subset
$$X_c = \set{ (θ,I,s,S) \mid Θ ∈ \T^n, s ∈ (c,1/c), I=0, S=0}.$$

\begin{lemma}
  \label{lem:normal-form-near-T}
  Let $𝐇 $ be the Nos{é}-thermostated
  hamiltonian~\eqref{eq:nose} and let
  $\thermostateqm{T}$ be the thermostatic equilibrium
  set~\eqref{eq:thermo-eq}. Then, there are neighbourhoods
  $W ⊃ \thermostateqm{T}$, $X ⊃ X_c$, $X ⊂ \cotangent(\T^n × \R)$,
  a submanifold $\thermostateqm{\hat{T}} ⊂ W$ that is a graph over
  $\thermostateqm{T}$ and a real-analytic symplectomorphism $φ$ such
  that
  \[
    \xymatrix{(θ_1,I_1,s,S) = φ(θ,I,v,V), &  & X \ar@{->}[r]^{φ}          & W \\
                                          &  & X_c \ar@{->}[r]^{φ} \ar[u] & \thermostateqm{\hat{T}} \ar[u]
    }
  \]
  and
  \begin{align}
    \label{eq:normal-form-near-T-h}
    𝐇 ⎄ φ & = \overbrace{\hat{G}(I;1/s_o(I)) + T \ln s_o(I) + ½ ε α(I) \left( v ² + V ² \right)}^{\mean{𝐇}  ⎄ φ} + ε^{\frac{3}{2}} 𝐏_{ε} ⎄ φ.
  \end{align}
  where $\hat{G}(I;σ)$ is the pullback of $H_{σ}$ in action-angle
  coordinates, $α$ is defined in \eqref{eq:gamma-alpha}, $ε=√ a$,
  $s_o$ is defined in~\eqref{eq:hats-0}, $\mean{𝐇}$ is an integrable truncation
  of $𝐇$ and $𝐏_{ε}$ is a remainder term that is real analytic in $ε$
  for $ε>0$ and continuous at $ε=0$.
\end{lemma}
\begin{proof}
  Let $V ⊂ L$ be a toroidal cylinder with angle-action variables
  $(\hat{θ},\hI) : V → \T^n × B$ for $H$. By
  proposition~\ref{prop:hlam-gen-fun}, for each such $V ⊂ L$ and each
  $σ$ sufficiently close to $1$, there is a symplectic map
  $f_{σ} : V → L$ such that $H_{σ} ⎄ f_{σ}$ is a real-analytic
  function of the action variables $\hI$ and the parameter $σ$,
  i.e. $H_{σ} ⎄ f_{σ}(\hat{θ},\hI) = \hat{G}(\hI;σ)$ where $\hat{G}$
  is real-analytic in both the action variable and parameter $σ$ and
  $f_1$ is the identity. It follows that there is a real-analytic
  generating function $\hat{θ} · I + Σ(\hat{θ},I;σ)$ of the
  symplectomorphism $f_{σ}$ that satisfies
  \begin{align}
    \label{eq:normal-form-gen-fun}
    (θ,I)     & = f_{σ}(\hat{θ},\hI)                 &  & \textrm{iff} & \hI              & = I+\didi[{Σ}]{\hat{θ}}, & θ & = \hat{θ}+\didi[{Σ}]{I}, \\\notag
              &                                        &  & \textrm{and} & Σ(\hat{θ},I;σ) & = O(σ-1).
  \end{align}
  Let us use the symplectomorphism $f_{σ}$ to define a
  symplectomorphism $f$ of $V × \cotangent \R^+$ as follows: the
  generating function of $f$ is defined by
  \begin{equation}
    \label{eq:normal-form-gen-fun-ext}
    ν(\hat{θ},I,s,\hat{S}) = \hat{θ} · I + Σ(\hat{θ},I;1/s) + \hat{S} s.
  \end{equation}
  Hence,
  \begin{align}
    \label{eq:normal-form-gen-fun-ext-trans}
    (θ,I,s,S) & = f(\hat{θ},\hI,\hat{s},\hat{S}) &  & \textrm{iff} & S &= \hat{S} - \dfrac{1}{s^2} \, \didi[{Σ}]{σ}, & \hat{s} &= s,
  \end{align}
  and \eqref{eq:normal-form-gen-fun} holds. Therefore,
  \begin{equation}
    \label{eq:normal-form-H-f}
    𝐇 ⎄ f(\hat{θ},\hI,\hat{s},\hat{S}) = \hat{G}(\hI ; 1/\hat{s}) + \dfrac{a}{2} \left( \hat{S} - \dfrac{1}{\hat{s}^2} \, \hat{Σ}_3 \right)^2 + T \ln \hat{s}
  \end{equation}
  where $\hat{Σ}$ is $Σ$ evaluated at $(\hat{θ},I;σ)$ with $I$
  a function of $\hat{θ}$ and $\hI$ by
  \eqref{eq:normal-form-gen-fun} and
  $\hat{Σ}_3 = \didi[\hat{Σ}]{σ}$ is the partial derivative with
  respect to the third variable.

  Let us define the set of thermostatic equilibria,
  $\thermostateqm{\hat{T}}$, for the function $\hat{G}$. Let
  $$ \thermostateqm{\hat{T}} = \set{(\hat{θ},\hat{I},\hat{s},\hat{S}) \mid \hat{G}_2(I;1/\hat{s}) = \hat{s} T, \hat{s}>0, \hat{S}=0 }, $$
  where $\hat{G}_2$ is the partial derivative of $\hat{G}$ with
  respect to the second variable. The function
  $\hat{κ} := \hat{G}_2(\hI;1/\hat{s}) / \hat{s}$ is the orbit mean
  temperature of $H$ in this system of coordinates. By hypothesis and
  since the map $f$ is exact symplectic, by
  remark~\ref{rem:invariance-of-orbit-mean-temp}, when $a=0$ the
  function $\hat{κ}$ equals $κ ⎄ f$ and $T$ is a regular value. Hence,
  for all $a$ sufficiently small there are neighbourhoods
  $\hat{J} ∋ 1$ and $\hat{W} ⊃ \hat{κ}^{-1}(T)$ and a real-analytic
  function $\hat{s}_o : \hat{W} → \hat{J}$ such that
  \begin{align}
    \label{eq:hats-0}
    (\hI,\hat{s}) ∈ \hat{W} × \hat{J}, \quad \hat{κ}(\hI;1/\hat{s}) = T \qquad\iff\quad \hat{s}=\hat{s}_o(\hI).
  \end{align}
  
  Define a generating function $\hat{ν}$ by
  \begin{equation}
    \label{eq:normal-form-gen-fun-ext2}
    \hat{ν}(\tilde{θ},\hI,\tilde{u},\hat{S}) = \tilde{θ} · \hI + \dfrac{\hat{S} \hat{s}_o(\hI)}{1-\tilde{u}} 
  \end{equation}
  and symplectomorphism $\hat{f}$ by
  \begin{align}
    \label{eq:normal-form-gen-fun-ext-trans2}
    (\hat{θ},\hI,\hat{s},\hat{S})                      & = \hat{f}(\tilde{θ},\tI,\tilde{u},\tilde{U}) &  & \textrm{iff} & \hat{S} & = \tilde{U} (1-\tilde{u})^2/s_o(\tI), & \hat{s} & = \hat{s}_o(\tI)/(1-\tilde{u}) \\\notag
                                                       &                                              &  &              & \hI     & = \tI,                                & \hat{θ} & = \tilde{θ} + \didi[{\hat{ν}}]{\hI}.
  \end{align}
  This produces
  \begin{align}
    \notag
    & 𝐇 ⎄ f ⎄ \hat{f}(\tilde{θ},\tI,\tilde{s},\tilde{S}) \\\notag
    & = \hat{G}(\tI ; (1-\tilde{u})/\hat{s}_o(\tI)) + \dfrac{a}{2} \left( \tilde{U}(1-\tilde{u})^2/\hat{s}_o(\tI) - \left( \dfrac{1-\tilde{u}}{\hat{s}_o(\tI)} \right)^2 \, \tilde{Σ}_3 \right)^2 \\\label{eq:normal-form-H-f2-full}
                                                       & \phantom{=}    + T \ln \hat{s}_o(\tI) - T \ln(1-\tilde{u}), \\\notag
    &= \hat{G}(\tI;1/\hat{s}_o) + T \ln \hat{s}_o + ½ \tilde{u}^2 \, \left( \dfrac{\hat{G}_{22}(\tI;1/\hat{s}_o)}{\hat{s}_o^2} + T \right) + \dfrac{a}{2 \hat{s}_o^2} \tilde{U}^2 + T O_3(\tilde{u}) \\\label{eq:normal-form-H-f2}
    &\phantom{=} + \dfrac{a}{2} \, \left( \dfrac{1-\tilde{u}}{\hat{s}_o} \right)^4 \, \tilde{Σ}_3 \, \left( - 2 \tilde{U} \hat{s}_o + \tilde{Σ}_3 \right),
  \end{align}
  where $T O_3(\tilde{u})=T(\tilde{u}^3/3+\cdots)$ is the collection
  of terms of degree $3$ and higher coming from last term
  in~\eqref{eq:normal-form-H-f2-full}.

  Next, define a generating function
  \begin{align}
    \label{eq:normal-form-gen-fun-ext3}
    \tilde{ν}(\check{θ},\tI,\check{v},\tilde{U}) & = \check{θ} · \tI + γ(\tI) \check{v} \tilde{U},          &  &  & \textrm{ where } \\
    \label{eq:gamma-alpha}
    γ(\tI)                                       & = \sqrt[4]{ \dfrac{a}{ \hat{G}_{22} + \hat{s}_o^2 T } }, &  &  & α(\tI) & = \dfrac{\sqrt{ \hat{G}_{22} + \hat{s}_o^2 T }}{\hat{s}_o^2},
  \end{align}
  The resulting symplectomorphism, $\tilde{f}$,
  transforms~\eqref{eq:normal-form-H-f2} into the Hamiltonian
  $𝐆=𝐆_{ε,T}$ where (with $ε = \sqrt{a}$ and dropping the $\check{\ }$
  decoration on the variables)
  \begin{align}
    \label{eq:normal-form-H-f3}
    𝐆                                            & = \underbrace{\hat{G}(I;1/\hat{s}_o) + T \ln \hat{s}_o(I) + ½ ε α(I) \left( v^2 + V^2 \right)}_{\mean{𝐆}} + ε^{3/2} \, P(I,θ,v,V;ε),
  \end{align}
  where $P$ is real-analytic in its arguments, real-analytic in $ε$
  for $ε>0$ and continuous at $ε=0$.

  The remainder of the proof of the lemma follows from defining
  $φ = f ⎄ \hat{f} ⎄ \tilde{f}$, whence $𝐆 = 𝐇 ⎄ φ$ and $P = 𝐏 ⎄ φ$.
\end{proof}

\begin{remark}
  \label{rem:normal-form-near-T}
  The reason the proof of Lemma~\ref{lem:normal-form-near-T} works is
  ultimately due to the fact that each $H_{σ}$ is completely
  integrable, or in somewhat less precise terms, the Nos{é}-thermostat
  ``vibrates'' an integrable hamiltonian within a family of integrable
  hamiltonians. One can see that the preceding lemma readily admits
  the following generalization:

  Let $N_T(s,S) = F(s,S) + T \ln s$ be a generalized variable-mass
  thermostat of order
  $2$~(definition~\ref{def:order-2-thermostat-generalized}) and define
  $N_{a,T}(s,S) = a^{-2} N_T(s,aS)$ for $a> 0$. If $𝐇$ is thermostated
  by $aN_{a,T}$, then the conclusions of
  Lemma~\ref{lem:normal-form-near-T} hold as stated with the only
  exception that $α$ equals $\sqrt{Ω_2(\hat{s}_o)}$ times the $α$
  appearing in~\eqref{eq:gamma-alpha}.
  %% (and $γ$ equal to $\sqrt[4]{Ω_2(\hat{s}_o)}$ times the $γ$ in
  %% that same line)

  In this case, one sees that the remainder term $𝐏_{ε}$ absorbs all
  but the lowest degree term ($½ a Ω_2(s) S^2$) in the thermostat, and
  when we expand around the thermostatic equilibrium set, only the
  lowest degree term ($½ a Ω_2(\hat{s}_o) (\tilde{U}/\hat{s}_o)^2$) is
  not absorbed in the remainder.
\end{remark}

\begin{remark}
  \label{rem:normal-form-near-T-nose}
  To lowest order, the frequency of the normal oscillations of
  $\mean{𝐇} | \thermostateqm{\hat{T}}$ is $ε α(I)$ where $ε = √ a$. In
  \cite[eq. 2.29]{doi:10.1080/00268978400101201}, Nos{é} derives an
  approximation to the frequency of the oscillations of the thermostat
  state $s$. In his solution, one finds the approximation to this
  frequency to be, in the notation of the current paper,
  \[
    ε \, \left( \dfrac{2 n T }{(n+1) s_o^2 } \right)^{½},
  \]
  In comparison to \eqref{eq:gamma-alpha}, Nos{é}'s approximation
  imputes the value $\hat{G}_{22} = nT/(n+1) s_o^2$ along the
  thermostatic equilibrium set. On the other hand, it follows from the
  calculations in remarks~\ref{rem:normal-form-near-T-example} \&
  \ref{rem:normal-form-near-T-example-contd}, that
  $\hat{G}_{22} = (λ-1) T \, s_o^2$ on that same set when $H = G(I)$
  is positively homogeneous of degree $λ$ in $I$, so Nos{é}'s
  approximation is only correct when the degree of homogeneity
  $λ = 2 - 2/(n+1)$. Leimkuhler \& Sweet~\cite[p. 190]{MR2136523} use
  Nos{é}'s approximation to determine an ``optimal'' choice of mass
  $Q$ ($=1/a$) to thermostat a single harmonic oscillator, but
  they find that it produces an unsatisfactory distribution of the
  thermostat state and need to make $Q$ almost an order of magnitude
  smaller to produce a satisfactory distribution.

  Figures~\ref{fig:nose} \& \ref{fig:nose-1e-1} show the ratio of the
  frequencies of the normal and internal oscillations of a $1$-degree
  of freedom hamiltonian of the form
  of~\eqref{eq:normal-form-near-T-homog} that is thermostated with the
  Nos{é} thermostat with $a=10^{-2}$ and $10^{-1}$, respectively, and
  $T=1$. It is clear from the graphs that Nos{é}'s approximation is
  only accurate for the thermostated harmonic oscillator, while the
  present paper's approximation is accurate in all the examples
  considered.

  \begin{figure}[p]
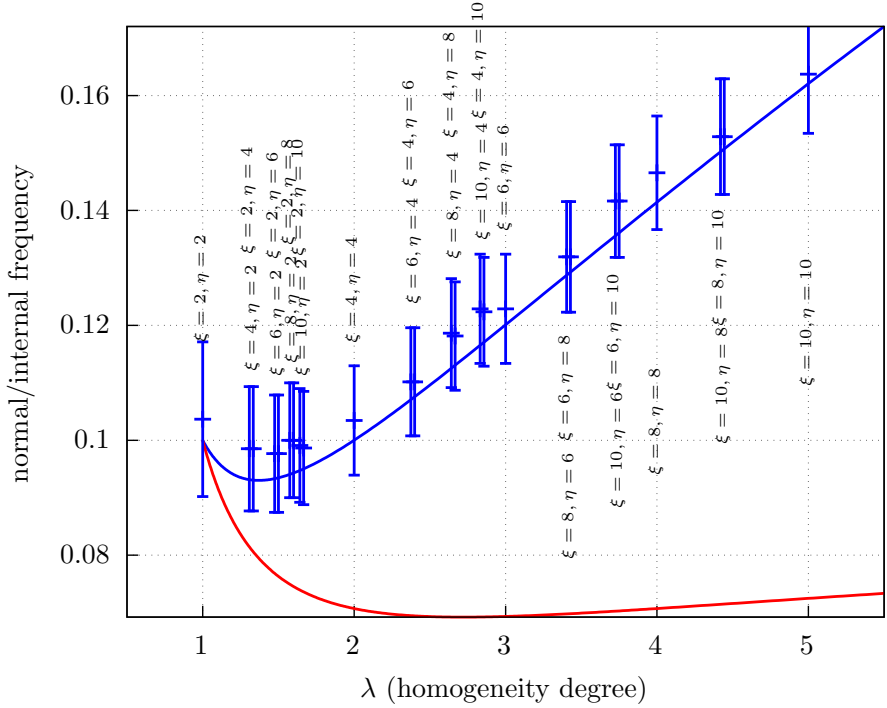

    \centering
    \caption{(T) The ratio of the frequencies of the normal (thermostat)
      and internal oscillations of a weighted homogeneous, single
      degree-of-freedom hamiltonian. The upper curve (blue) is that
      implied by Lemma~\ref{lem:normal-form-near-T} while the lower
      curve (red) is implied by Nos{é}'s approximation (see
      Remark~\ref{rem:normal-form-near-T-nose}). The dots are
      determined by integrating the Nos{é}-thermostated
      weighted-homogeneous
      hamiltonians~\eqref{eq:normal-form-near-T-homog} with even
      integer exponents $2 ≤ ξ, η ≤ 10$. The computation of the error
      bar is described in~figure~\ref{fig:nhnd-freq-resp-256}. The
      orbits are started at $x=0, S=0, s=1$ and $p_x$ is determined from
      the condition that $H=T/λ$ (see
      remark~\ref{rem:normal-form-near-T-example}). A stepsize of
      $h=2^{-5}$ and time interval of length $2^{10}$ is used;
      $a=10^{-2}$ and $T=1$ are the fixed thermostat
      parameters. (ML+R) The Fourier transforms of $x,p_x, √ a S$ and
      $s$ (ignoring the mean value of $s$) for the Nos{é}-thermostated
      harmonic oscillator ($ξ=2=η$). (BL+R) As above, with $ξ=10=η$. }
    \label{fig:nose}
    \ltxfigure{nhnd/nhnd-nose-freq-rat-mac.tex}{\textwidth}{!} \\
    \ltxfigure{nhnd/nhnd-nose-freq-rat-x-p-S-xi=2-eta=2-mac.tex}{0.45\textwidth}{!}
    \ltxfigure{nhnd/nhnd-nose-freq-rat-s-S-xi=2-eta=2-mac.tex}{0.45\textwidth}{!} \\
    \ltxfigure{nhnd/nhnd-nose-freq-rat-x-p-S-xi=10-eta=10-mac.tex}{0.45\textwidth}{!}
    \ltxfigure{nhnd/nhnd-nose-freq-rat-s-S-xi=10-eta=10-mac.tex}{0.45\textwidth}{!}
  \end{figure}
  \begin{figure}[p]
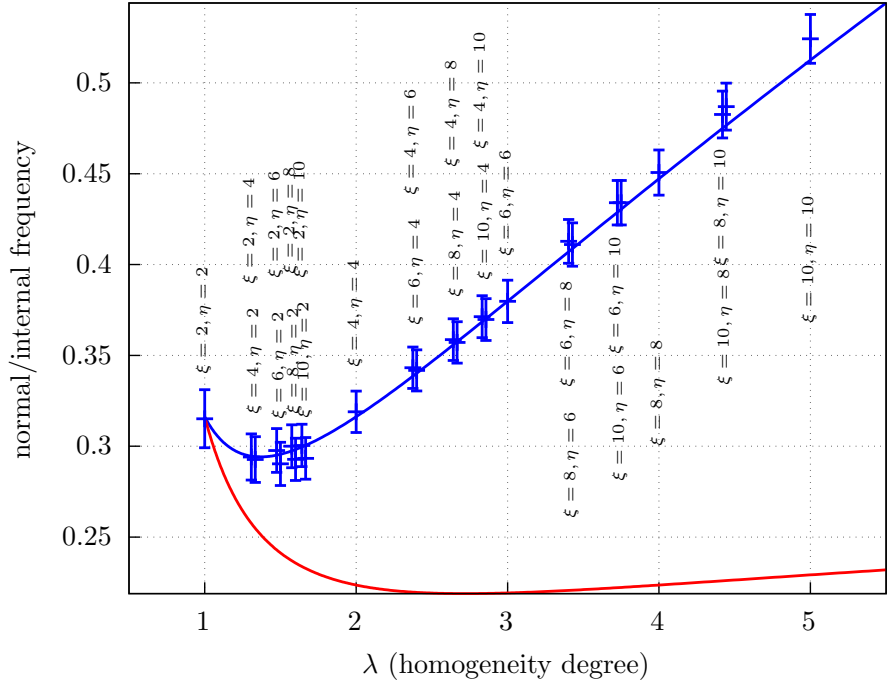

    \centering
    \caption{(T,ML+R) As for figure~\ref{fig:nose}, but with
      $a=10^{-1}$. (BL+R) As above, with $ξ=10, η=6$. The dominant
      frequency of the oscillations in $x$ (and $p_x$) is visible at
      $k=235$ while there are ``beats'' at $k=134, 335$; the secondary
      peak at $k=702$ is surrounded by two pair of beat frequencies at
      $k=803,904$ and $k=601,500$. The dominant frequency in $S$ is at
      $k=468$ and is driven by the internal oscillations of $x$ and
      $p_x$. It is surrounded by two pair of beats at $k=569,670$ and
      $k=266,367$. The frequency of the normal oscillations of $S$
      ($k=102$) is responsible for the beating.}
    \label{fig:nose-1e-1}
    \ltxfigure{nhnd/nhnd-nose-freq-rat-a=1e-1-mac.tex}{\textwidth}{!} \\
    \ltxfigure{nhnd/nhnd-nose-freq-rat-x-p-S-xi=2-eta=2-a=1e-1-mac.tex}{0.45\textwidth}{!}   
    \ltxfigure{nhnd/nhnd-nose-freq-rat-s-S-xi=2-eta=2-a=1e-1-mac.tex}{0.45\textwidth}{!} \\
    \ltxfigure{nhnd/nhnd-nose-freq-rat-x-p-S-xi=10-eta=10-a=1e-1-mac.tex}{0.45\textwidth}{!}   
    \ltxfigure{nhnd/nhnd-nose-freq-rat-s-S-xi=10-eta=10-a=1e-1-mac.tex}{0.45\textwidth}{!}
  \end{figure}
\end{remark}

\begin{remark}
  \label{rem:normal-form-near-T-example}
  Let us compute an example to illustrate
  Lemma~\ref{lem:normal-form-near-T}. Let the number of degrees of
  freedom be $n ≥ 1$, let $ξ,η > 0$ be even integers and $λ>0$ satisfy
  the identity $1/λ = 1/ξ + 1/η$. Define
  \begin{align}
    \label{eq:normal-form-near-T-homog}
    H(q,p) &= h(q_1,p_1) + \cdots + h(q_n,p_n), &&& h(x,p_x) &= \dfrac{1}{c} \left( p_x^{ξ} + x^{η} \right).
  \end{align}
  (the constant $c>0$ is a normalization constant chosen below). The
  hamiltonian $H$ is separable and a simple computation shows that in
  action-angle variables
  \begin{align}
    \label{eq:normal-form-near-T-homog-aa}
    H = G ⎄ I &= I_1^{λ} + \cdots + I_n^{λ},
  \end{align}
  where the normalization constant $c$ is chosen so that
  $c = \left( π η /(2 Β(1+1/ξ,1/η)) \right)^{λ}$ and $Β$ is the Beta
  function~\cite[Chapter 6]{MR1225604}.

  It is easy to see that $H_{σ}(q,p) = σ^{λ} \, H(Q,P)$ when $q=Q
  σ^{χ}, p = P/σ^{χ}$ and $χ = λ/ξ$. The generating function
  \begin{equation}
    \label{eq:normal-form-gen-fun-1}
    ν(Q,p,s,\hat{S}) = Q · p · s^{-χ} + \hat{S} s
  \end{equation}
  induces the symplectomorphism $(q,p,s,S) = f(Q,P,\hat{s},\hat{S})$
  \begin{equation}
    \label{eq:normal-form-f-1}
    (q,p,s,S) = (Q s^{-χ}, P s^{χ}, \hat{s}, \hat{S} - χ \hat{s}^{-1} Q P).
  \end{equation}
  Therefore, the Nos{é}-thermostated hamiltonian is transformed to
  \begin{equation}
    \label{eq:normal-form-H-f-1}
    𝐇 ⎄ f(Q,P,\hat{s},\hat{S}) = \hat{s}^{-λ} H(Q,P) + \dfrac{a}{2} \left( \hat{S} - χ \hat{s}^{-1} Q P \right)^2 + T \ln \hat{s}.
  \end{equation}
  With the convenient notation that $I^{λ} = \sum_{i=1}^n I_i^{λ} = G(I)$, one sees that the function $\hat{G}$ equals
  \begin{equation}
    \label{eq:normal-form-hatG-1}
    \hat{G}(\hat{I};1/\hat{s}) = G(\hI/\hat{s}) = (\hat{I}/\hat{s})^{λ}.
  \end{equation}
  The thermostatic equilibrium set $\thermostateqm{\hat{T}} = \hat{G}^{-1}(T/λ)$ and so
  \begin{equation}
    \label{eq:normal-form-s0}
    \hat{s}_o(\hI) = \left( \dfrac{λ G(\hI)}{T} \right)^{\dfrac{1}{λ}}.
  \end{equation}
\end{remark}

\subsection{Main Theorem}
\label{sec:main-theorem}

The re-scaled frequency map of
$\mean{𝐆} = \mean{𝐇} ⎄ φ$~\eqref{eq:normal-form-H-f3} from
Lemma~\ref{lem:normal-form-near-T} is computed to be
\begin{align}
  \label{eq:rescaled-frequency-map}
  Ω(I) & = (\hat{G}_1(I;1/\hat{s}_o),\, α(I)), &  &  & α(I) & =  \dfrac{\sqrt{Ω_2(\hat{s}_o) \, ( \hat{G}_{22} + \hat{s}_o^2 T )}}{\hat{s}_o^2},
\end{align}
using equation~\eqref{eq:hats-0} and
remark~\ref{rem:normal-form-near-T}.
\begin{theorem}
  Assume that for some $T>0$, the re-scaled frequency map $Ω$ of the
  function $𝐆_{ε,T}$~\eqref{eq:normal-form-H-f3} is R-non-degenerate
  when $ε=0$. Then, there exists a function $ε_o(T)$ such that $ε_o$
  is positive for all but countably many values of $T>0$ and for all
  $0 < ε < ε_o(T)$, there exists a neighbourhood of the thermostatic
  equilibrium set $\hat{\thermostateqm{T}}$ and a positive-measure,
  $𝐆_{ε,T}$-invariant set $K_{ε,T}$ in this neighbourhood. In
  particular, the hamiltonian flow of $𝐆_{ε,T}$ is not ergodic for
  $0 < ε < ε_o(T)$.
\end{theorem}
\begin{proof}
  By Lemma~\ref{lem:normal-form-near-T} and
  Theorem~\ref{thm:modified-chierchia-pusateri}, the theorem follows.
\end{proof}

\begin{remark}
  \label{rem:normal-form-near-T-example-contd}
  (Continuation of remark~\ref{rem:normal-form-near-T-example}). It is
  straightforward to verify that, when $λ ≠ 1$, the re-scaled
  frequency map $Ω(I)$~\eqref{eq:rescaled-frequency-map} is
  R-non-degenerate for this example. Indeed, one can verify that the
  map $I → \hat{G}_1(I;1/\hat{s}_o)$ is a diffeomorphism onto its
  image: This is because it equals the diffeomorphism $I → \D{G}(I)$
  multiplied by a positive scalar ($1/\hat{s}_o^{λ}$) and the degree
  of homogeneity of their product is $-1$. Therefore, if the image of
  $Ω$ is contained in a linear subspace, there is a normal vector to
  the subspace of the form $v=(w,-1)$ and so
  $α(I) = \ip{w}{\hat{G}_1(I;1/\hat{s}_o)}$ for a fixed vector $w$. In
  this case, $α(I) = \sqrt{Ω_2(\hat{s}_o)} \, \sqrt{λ T} / \hat{s}_o$
  so the subspace condition implies the identity in $I$:
  \begin{equation}
    \label{eq:omega2}
    \sqrt{Ω_2(\hat{s}_o)} = \sqrt{ T/λ } \, \dfrac{1}{G(I)} \, \sum_{i=1}^n w_i I_i^{λ-1}.
  \end{equation}
  Evaluation of each side at the points $I^{(i)}=δ_{ij}$ ($δ$ is
  Kronecker's delta function), implies that
  $w_i = w_o = \sqrt{Ω_2(\hat{s}_{*})} (λ/T)^{½ + 1/λ}$ where
  $\hat{s}_{*} = (λ/T)^{1/λ}$. Substituting this into
  \eqref{eq:omega2} implies
  \begin{equation}
    \label{eq:omega3}
    \sqrt{Ω_2(\hat{s}_o)} = \sqrt{Ω_2(\hat{s}_{*})} · λ^{2/λ} · T^{-2} · \sum_{i=1}^n I_i^{λ-1}.
  \end{equation}
  Let $I = b s I^{(i)}$ for $b = (T/λ)^{1/λ}$. Then $G(I)= (bs)^{λ}$ and
  $\hat{s}_o(I) = s$ so \eqref{eq:omega3} implies that
  \begin{equation}
    \label{eq:omega4}
    Ω_2(s) = Ω_2(\hat{s}_{*}) · (λ/T)^{2(1+1/λ)} · s^{2(λ-1)} = c s^{2λ-2}.
  \end{equation}
  On the other hand, if $I = bs \sum_{i=1}^k I^{(i)}$ for $k ≤ n$,
  then $G(I) = k (bs)^{λ}$ and $\hat{s}_o(I) = k^{1/λ} s$, which
  implies that $Ω_2(s) = k^{1+2/λ} · c · s^{2λ-2}$. Thus, when $n>1$,
  there is no solution to \eqref{eq:omega3}, hence the re-scaled
  frequency map $Ω$ is R-non-degenerate.

  This demonstrates R-non-degeneracy for the thermostated hamiltonian
  except when $λ=1$ (the thermostated harmonic oscillator) or
  $n=1$. When $n=1$ and $λ ≠ 1$, an R-degenerate generalized
  variable-mass thermostat of order $2$ must have the variable mass
  function $Ω_2$ described by \eqref{eq:omega4}. In particular, the
  Nos{é} and logistic thermostats are R-non-degenerate in this case,
  too, while the generalized Winkler thermostat is R-non-degenerate
  unless $λ = 2 - 1/\e$~(see~p.~\pageref{it:winkler}).
\end{remark}

\subsection{Main Corollary}
\label{sec:main-cor}

Let $X$ be a real-analytic manifold and $C^{k}(X)$ be the space of
real-valued functions $f : X → \R$ which are continuous with $k ≥ 0$
continuous derivatives; $C^{∞}(X) = \bigcap_{k ≥ 0} C^{k}(X)$ and
$C^{ω}(X) ⊂ C^{∞}(X)$ is the set of real-analytic functions. Besides
the intrinsic direct limit topology, there are many topologies on the
space of real-analytic functions $C^{ω}(X)$ since
$C^{ω} ⊂ C^{∞} ⊂ C^{k}$ for all $k ≥ 0$. Let us use the subspace
topology from the uniform $C^{k}$ topology. In this case, a basic open
set consists of a compact set $K ⊂ X$ and an open set
$U ⊂ \oplus_{i=0}^k \R^{d_i}$ such that the map $x \mapsto J_k(x)$,
$J_k(x) = (f(x),\D{f}_x, \ldots, \D{}^kf_x)$, maps $K$ into $U$ ($d_i$
is the dimension of the linear space of symmetric tensors of degree
$i$ on $\R^n$). When $X$ is symplectic, there is a distinguished
subset $𝔍$ of completely integrable hamiltonian functions and $𝔍^k = 𝔍
∩ C^{k}(X)$ for $k ∈ \set{2,\ldots,∞,ω}$. We can equip $𝔍^{ω}$ with
the $C^{k}$ subspace topology. 

\begin{corollary}
  \label{cor:main-cor}
  Let $H ∈ 𝔍^{ω}$ be completely integrable and $N=N_{a,T}$ be a
  generalized variable-mass thermostat of order $2$. Assume that
  \begin{enumerate}
  \item $T$ is a regular value of the orbit mean temperature function
    $κ$ of $H$;
  \item The function $\hat{G}_{22}$ is non-negative.
  \end{enumerate}
  Then, there is a set $𝔘 ⊂ 𝔍^{ω}$ that is relatively open in the
  uniform $C^2$ topology such that
  \begin{enumerate}
  \item $H$ is in the closure of $𝔘$;
  \item If $H' ∈ 𝔘$, then the re-scaled frequency map $Ω'$ of $𝐇'$ is
    R-non-degenerate.
  \end{enumerate}
\end{corollary}

\begin{proof}
  The proof of this corollary is straightforward. Consider a
  ``semi-global'' perturbation of $\hat{G}$ on $W × \R^+$ that is of
  the form $\hat{g}(I,s) = ½ (s-\hat{s}_o)^2 Φ(I)$ where $Φ$ is a
  non-negative, real-analytic function and let
  $\hat{G}^{(η)} = \hat{G} + η \hat{g}$ where $η$ is a
  parameter. Assume that $Φ$ does not satisfy any linear equation with
  coefficients in the ring over $\R$ generated by $Ω_2(\hat{s}_o)$ and
  the components of $\hat{G}_1$. The functions $\hat{G}^{(η)}$ share
  the same thermostatic equilibrium scaling function $\hat{s}_o$ by
  construction. By choosing $Φ$ appropriately, it is clear that one
  may assume that $\hat{G}^{(η)}_{11}$ is non-degenerate at some point
  in $\thermostateqm{\hat{T}}$ for all $η > 0$ sufficiently
  small. Thus, suppose that for each $η$ there is a vector
  $v_{η} = (w_{η},-1)$ such that the perturbed, re-scaled frequency
  map $Ω^{(η)}$ lies in the linear subspace orthogonal to
  $v_{η}$. Then
  \begin{equation}
    \label{eq:r-non-degen-fun-rel}
    Ω_2(\hat{s}_o) = \dfrac{ \hat{s}_o^4 \, \left( \ip{w_{η}/η}{\hat{G}_1} \right)^2 }{ \hat{G}_{22}/η + \hat{g}_{22} + \hat{s}_o^2 T/η },
  \end{equation}
  must hold identically for $η > 0$. This implies that, as $η → ∞$,
  $\ip{w_{η}/η}{\hat{G}_1}$ must converge to a function of $I$; and
  so, since $\hat{G}_1$ is a local diffeomorphism, $w_{η}/η$ converges
  to a fixed vector $r$ as $η → ∞$. Therefore,
  \begin{equation}
    \label{eq:r-non-deg-pert}
    Φ = \dfrac{ Ω_2(\hat{s}_o) }{ \ip{r}{\hat{G}_1}^2 }.
  \end{equation}
  Thus $Φ$ satisfies a linear equation in the ring over $\R$ generated
  by $Ω_2(\hat{s}_o)$ and the coefficients of
  $\hat{G}_1$. Contradiction.

  Finally, it has been shown that there are specific perturbations
  that are ``semi-global'' and R-non-degenerate. For a global
  perturbation of $𝐇$, one can take $𝐇 + ψ(F ⎄ φ_{1/s}, s)$ where $ψ$
  is real-analytic, $F$ is the first-integral map of $H$, and $ψ$ is
  chosen as a sufficiently close approximation to a $\hat{g}$ from the
  preceding paragraph.
\end{proof}

\section{An example: rotationally-symmetric potentials}
\label{sec:llm-ex}

Let's consider examples of rotationally-symmetric mechanical
Hamiltonians on a surface $Σ =\R^2$ or
$\sphere 2$. In~\cite[section~4]{MR2519685}, Legoll, Luskin and
Moeckel consider the case where the hamiltonian is
\begin{equation}
  \label{eq:llm-rot-inv-ham}
  H(r,p_r,θ,p_{θ}) = ½ \left( p_r^2 + r^{-2} p_{θ}^2 \right) + v(r)
\end{equation}
in symplectic polar coordinates on $\cotangent \R^2$.

In addition to the general result on the existence of integrals for
the averaged Nos{é}-Hoover thermostat (see Theorem~\ref{thm:llm-ave}
above), for the specific potential $v(r)=r^2+r^4$, they show numerical
evidence that, for ``small'' $a \sim 10^{-2}$ and
$T=1$,\footnote{Recall that $a$ in the present paper is $1/Q$ in
  \cite{MR2519685} and $ε^2$ in~\eqref{eq:nose-hoover}.} the
Nos{é}-Hoover thermostated system~\eqref{eq:nose-hoover} is nearly
integrable and the averaged system has $2$ independent integrals,
which corresponds to $3$ independent integrals for the
Nos{é}-thermostated hamiltonian $𝐇$~\eqref{eq:nose}.

\subsection{The case of $\R^2$}
\label{sec:r2}

Let's consider the case where $H$~\eqref{eq:llm-rot-inv-ham} is a
rotationally-invariant, real-analytic mechanical
hamiltonian. Thermostat $H$ with a real-analytic generalized,
variable-mass thermostat of order $2$, to get:
\begin{equation}
  \label{eq:llm-rot-inv-therm-r2}
  𝐇(r,p_r,θ,p_{θ},s,S) = ½ \left( p_r^2 + r^{-2} p_{θ}^2 \right)/s^2 + v(r) + F(s,aS)/a + T \ln s,
\end{equation}
where $F(s,S)$ is described
in~\eqref{eq:order-2-thermostat-generalized}.

Recall the two hypotheses from the statement of
Theorem~\ref{thm:llm-rot-inv} about the potential $v$ that are assumed
to hold on some open interval in $\R^+$:

\begin{enumerate}
\item[H1.] The function $v'(r) > 0$;
\item[H2.] The function $r v''(r) + v'(r) > 0$.
\end{enumerate}

\begin{lemma}
  \label{lem:llm-rot-inv-therm-cov}
  Assume H1 and H2 hold on an open interval $J ⊂ \R^+$.
  
  Let $\rs = \rs(p_{θ}/s;T)$ and $s_o = s_o(p_{θ};T)$ be
  real-analytic, scalar functions of a single variable parameterized
  by the temperature $T$. Define the generating function
  \begin{equation}
    \label{eq:llm-rot-inv-therm-gen-fun}
    ν(ρ,p_r,\hat{θ},p_{θ},u,S) = p_r · \rs · (1+ρ) + \hat{θ} · p_{θ} + S · s_o / (1-u),
  \end{equation}
  which generates the symplectic map $f$
  \begin{align}
    \label{eq:llm-rot-inv-therm-map}
    r     & = \rs · (1+ρ), &  &  & p_r & = p_{ρ} / \rs, &  &  & θ & = \hat{θ} + ξ,    \\\notag
    p_{θ} & = p_{\hat{θ}}  &  &  & s   & = s_o / (1-u), &  &  & S & = (1-u)^2 U / s_o \\\notag
          &                &  &  &     &                &  &  &   & - p_{ρ} · \rs' · (1+ρ) · (1-u)^2 / (\rs · s_o^2),
  \end{align}
  where $ξ$ is a real-analytic function and $\rs$ and its derivative
  are evaluated at $p_{θ}(1-u)/s_o$.

  Then, $𝐆 = 𝐇 ⎄ f$ has a relative critical point at $ρ=u=p_{ρ}=U=0$
  if there is an $\rs ∈ J$ such that the following hold:
  \begin{align}
    \label{eq:llm-rot-inv-condns}
    T & = \rs · v'(\rs), &  &  & τ & = |p_{θ}| / s_o, &  &  & \textrm{and } τ^2 & = \rs^3 · v'(\rs).
  \end{align}
\end{lemma}

% A note about the interpretation of \eqref{eq:llm-rot-inv-condns}. By
% hypothesis H1 and H2, the first equation can be viewed as defining a
% one-to-one {\em temperature}-like function, $r \mapsto r · v'(r)$, on
% $J$ such that $T$ is in its image; similarly, the third equation
% defines a one-to-one function $τ = τ(r)$ on $J$. These two equations,
% with $T$ fixed, determine a unique positive $τ=τ_o(T)$. The middle
% equation then implicitly defines $s_o = |p_{θ}|/τ_o(T)$.

\begin{proof}
  Let $U=p_{ρ}=u=ρ=0$. Since both $p_r$ and $S$ are linear in $p_{ρ}$
  and $U$ and $𝐇$ is quadratic in $p_r$ and $S$,
  $\dot{u}=0=\dot{ρ}$. Assume that $τ=p_{θ}/s_o$ and $p_{θ}>0$ (in
  case $p_{θ}<0$, one takes $τ=-p_{θ}/s_o$). Then one computes:
  \begin{align}
    \label{eq:llm-rot-inv-diffg}
    \dot{p}_{ρ} & = -𝐆_{ρ} = \left( τ/\rs \right)^2 - \rs · v'(\rs), \\\notag
    \dot{U}     & = -𝐆_u = \left( τ · \rs^3 · \rs' · v'(\rs) - τ^3 · \rs' - T \rs^3 + τ^2 · \rs \right)/s_o^3.
  \end{align}
  The first equation implies the last equation
  of~\eqref{eq:llm-rot-inv-condns} given the middle one. By hypotheses
  H1 \& H2, the function $r \mapsto r^3 · v'(r)$ is increasing on $J$
  and therefore the function $τ = \sqrt{ r^3 · v'(r) }$ is increasing
  on $J$. Hence, there is a local inverse $\rs = \rs(τ)$.

  The second equation of~\eqref{eq:llm-rot-inv-diffg} implies that
  \[
    T = \left( τ^2 · \rs - τ^3 · \rs' + τ · \rs^3 · \rs' · v'(\rs) \right)/\rs^3 = τ^2 /\rs^2 = \rs · v'(\rs),
  \]
  where the inverse function theorem has been applied to $\rs=\rs(τ)$
  to simplify the expression in parentheses. This proves the lemma.
\end{proof}

By virtue of the previous lemma, one can define the functions $τ(r) =
\sqrt{r^3 · v'(r)}$ and $T(r) = r · v'(r)$ for $r ∈ J$. By H1, $T(r) >
0$ and by H2, $T'(r) > 0$. Hence, there is a single-valued inverse
$r(T)$ for $T ∈ T(J)$. This implies that the
equations~\eqref{eq:llm-rot-inv-condns} determine a unique value for
$\rs = r(T)$ and a unique value $τ = τ(T)$. So, the middle equation
defines
\begin{equation}
  \label{eq:llm-rot-inv-s0}
  s_o(p_{θ};T) = \dfrac{|p_{θ}|}{τ(T)},
\end{equation}
and $\rs = \rs(τ(T);T)$, also.

Henceforth, it is assumed that $\rs$, $s_o$ and $τ$ are determined as
in lemma~\ref{lem:llm-rot-inv-therm-cov}.

\subsubsection{Symplectic reduction}
\label{sec:llm-rot-inv-sympred}

Let's fix the value $p_{θ}=μ≠0$ and reduce the hamiltonian $𝐆$ modulo
the rotational action by translation of $\hat{θ}$. The symplectic
reduction of $\set{p_{θ}=μ} ⊂ \cotangent (\R^2 × \R^+)$ by this free
action of $\SO{2}$ is a symplectic manifold $X_{μ}$ that is
symplectomorphic to $\cotangent(\R^+ × \R^+)$. The Darboux coordinates
$(ρ,p_{ρ},u,U)$ are defined on a neighbourhood of the reduced critical
point $(r=\rs,p_r=0,s=s_o,S=0)$ of the reduced hamiltonian $𝐆_{μ}$
(obtained from $𝐆$ by fixing $p_{θ}=μ$).

\begin{lemma}
  \label{lem:llm-rot-inv-d2g-red}
  The linearized hamiltonian vector field of $𝐆_{μ}$ at the critical point $ρ=u=p_{ρ}=U=0$ is
  \begin{equation}
    \label{eq:llm-rot-inv-d2g-red}
    \dot{𝐗} = 𝐀 \, 𝐗 =
    \begin{pmatrix}
      &&& A & B \\ &&& B & E \\ -C & \phantom{-}0 \\ \phantom{-}0 & -D
    \end{pmatrix} \, 𝐗
  \end{equation}
  where $𝐗 = [ρ,u,p_{ρ},U]$ and (with $r=\rs$, $W=(r v'' + 3 v')$)
  \begin{align}
    \label{eq:llm-rot-inv-sympred-d2g-coeffs}
    C & = r W,                                                          &  &  & D & = 2 r v' (W - 2 v')/W,  \\\notag
    A & = \dfrac{1}{(r s_o)^2} + 4a Ω_2(s_o) \dfrac{(v')^2}{s_o^2 W^2}, &  &  & B & = -2a Ω_2(s_o) \dfrac{v'}{s_o^2 W}, &  &  & E & = a Ω_2(s_o)/s_o^2.
  \end{align}
\end{lemma}

The characteristic polynomial of $𝐀$ is
\begin{equation}
  \label{eq:llm-rot-inv-char-poly}
  p(x) = x^4 + (DE+AC) x^2 + CD(AE-B^2).
\end{equation}
For $a>0$, the hessian of $𝐆_{μ}$ at the critical point is positive
definite, by H1 \& H2. Therefore, let $± i ω_1, ± i ω_2$ be the purely
imaginary roots of $p$, with $ω_2 ≥ ω_1 > 0$, and let $η =
ω_1/ω_2$. The function $η = η(μ; a, T)$ is continuous everywhere and
real-analytic except at the points where $η=1$.

\begin{lemma}
  \label{lem:llm-rot-inv-r-non-deg}
  The following hold:
  \begin{enumerate}
  \item If, for fixed $a, T > 0$, $η=η(μ; a, T)$ is a non-constant
    function of $μ=p_{θ}$, then there exists a full-measure set, $𝔐$,
    of $μ$ and a neighbourhood $O$ of the critical point of $𝐆_{μ}$
    such that for each $μ ∈ 𝔐$, there exists a positive-measure set of
    invariant tori $𝔗_{μ} ⊂ O$; or
  \item For each $T>0$ and $μ ≠ 0$, there exists a full-measure set
    $𝔄 ⊂ \R^+$ such that if $a ∈ 𝔄$, then there is a neighbourhood $O$
    of the critical point of $𝐆_{μ}$ and a positive-measure set of
    invariant tori $𝔗_{a} ⊂ O$.
  \end{enumerate}
\end{lemma}

To prove this lemma, one needs to recall a classic result of
Russmann. Recall that a vector $ω ∈ \R^n$ is Diophantine with constant
$γ > 0$ and exponent $τ > 0$ if
$$ \left| \ip{k}{ω} \right| ≥ γ |k|^{-τ}, \qquad \forall k ∈ \Z^n ∖ \set{0}.  $$
It is well-known that the set of Diophantine vectors is of full
measure in $\R^n$.

Russmann, in~\cite[p. 56]{MR213679}, proves that if the Hamiltonian
$H$ has a critical point with first Birkhoff invariant
$H_2 = \ip{ω}{I}$ where $ω$ is a Diophantine vector and the Birkhoff
normal form of $H$ is $B=\sum_{k=1}^{∞} a_k H_2^k$ (with $a_1=1$),
then this formal power series actually converges on a neighbourhood of
the critical point and there is a real-analytic symplectic map $φ$
defined on the same neighbourhood such that $H ⎄ φ = B$. In two
degrees of freedom, this implies the following.

\begin{theorem}[Russmann~\cite{MR213679}; Churchill, Pecelli,
  Sacolick and Rod~\cite{MR494256}]
  \label{thm:russmann0}

  $ $\newline
  \noindent
  Let $H(x_1,x_2,p_1,p_2) = H_2(I_1,I_2) + O(3)$ where
  $H_2 = ω_1 I_1 + ω_2 I_2$ and
  $I_i = ½ \left( x_i^2 + p_i^2 \right)$, be a real-analytic
  hamiltonian defined on a neighbourhood of $0 ∈ \R^4$. Assume that
  $ω=(ω_1,ω_2)$ is a Diophantine vector in $\R^2$. One of the two
  possibilities holds:
  \begin{enumerate}
  \item for some $k>1$, the Birkhoff normal form of $H$ of degree $k$
    is non-zero modulo polynomials of degree $k$ in $H_2$; or
  \item for all $k>1$, the Birkhoff normal form of $H$ of degree $k$
    is a polynomial in $H_2$.
  \end{enumerate}
  In the second case, there is a real-analytic symplectic map $φ$ defined
  on a neighbourhood of $0$ and real-analytic function $G$ such that
  $H ⎄ φ = G ⎄ H_2$.
\end{theorem}

\begin{remark}
  \label{rem:russmann0}
  It follows that if a two-degree-of-freedom hamiltonian $H$ has a
  critical point with $H_2 = \ip{ω}{I}$ and $ω$ Diophantine, then there
  is a neighbourhood of the critical point that contains a positive
  measure set of invariant tori. In case (1), this follows from the fact
  that there is a $k>0$ such that the Birkhoff polynomial $B_k$ of
  degree $k$ has a Hessian that is non-degenerate for some $(I_1,I_2)$
  near $(0,0)$. In case (2), it is clear since $H$ is integrable.
\end{remark}

\begin{proof}[Proof of Lemma~\ref{lem:llm-rot-inv-r-non-deg}]
  In case (1), $η$ is a non-constant function of $μ$, so the frequency
  map $μ → (ω_1,ω_2)$ is R-non-degenerate and real-analytic (except
  for at most countably many values of $μ$). It follows that the
  pre-image of the full-measure set of Diophantine vectors is a set of
  full measure.

  In case (2), one can see that the function $a → η(μ;a,T)$ is
  non-constant and so the argument of the previous paragraph applies.

  Remark~\ref{rem:russmann0} therefore implies the present Lemma.
\end{proof}

\begin{remark}
  \label{rem:om2}
  If $Ω_2$ is constant, then for fixed $a, T$, the coefficient
  $DE+AC \propto s_o^{-2}$ while $CD(AE-B^2) \propto s_o^{-4}$, so the
  roots of the characteristic
  polynomial~\eqref{eq:llm-rot-inv-char-poly} lie on a line through
  $0$ when $p_{θ}=μ$ varies. Thus, $η$ is constant as a function of
  $μ$. In addition, it is easy to see that when $Ω_2$ is non-constant,
  then $η$ is, too. On the other hand, case (2) of
  lemma~\ref{lem:llm-rot-inv-r-non-deg} holds independent of whether
  $Ω_2$ is constant or not. Figure~\ref{fig:freq-ratio} graphs the
  function $T → η(1; a,T)$ for the Nos{é}-thermostated planar system
  with potential energy $v(r)=r^2+r^4$ and selected values of $a$.
\end{remark}

\begin{figure}[p]
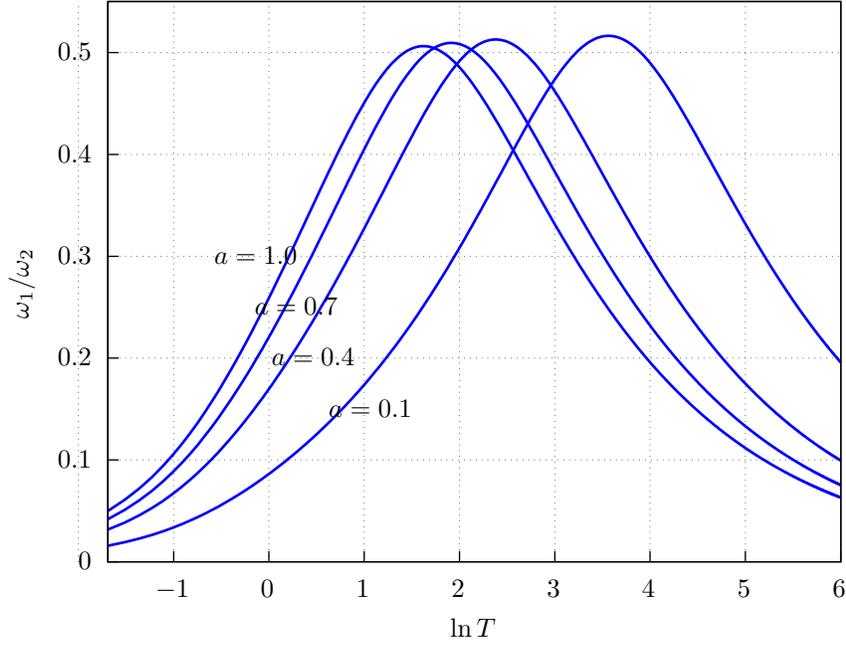

  \centering
  \caption{The frequency ratio $η=ω_1/ω_2$ v. the natural logarithm of
    the temperature $T$ for selected values of $a$. The potential
    energy is $v(r)=r^2+r^4$ and the thermostat is Nos{é}'s. The value
    of $η$ at $a=0$ is $T=0$ in all cases.}
  \label{fig:freq-ratio}
  \ltxfigure{nhnd/nhnd-freq-ratio-llm-mac.tex}{12cm}{!}
\end{figure}

\begin{figure}[p]
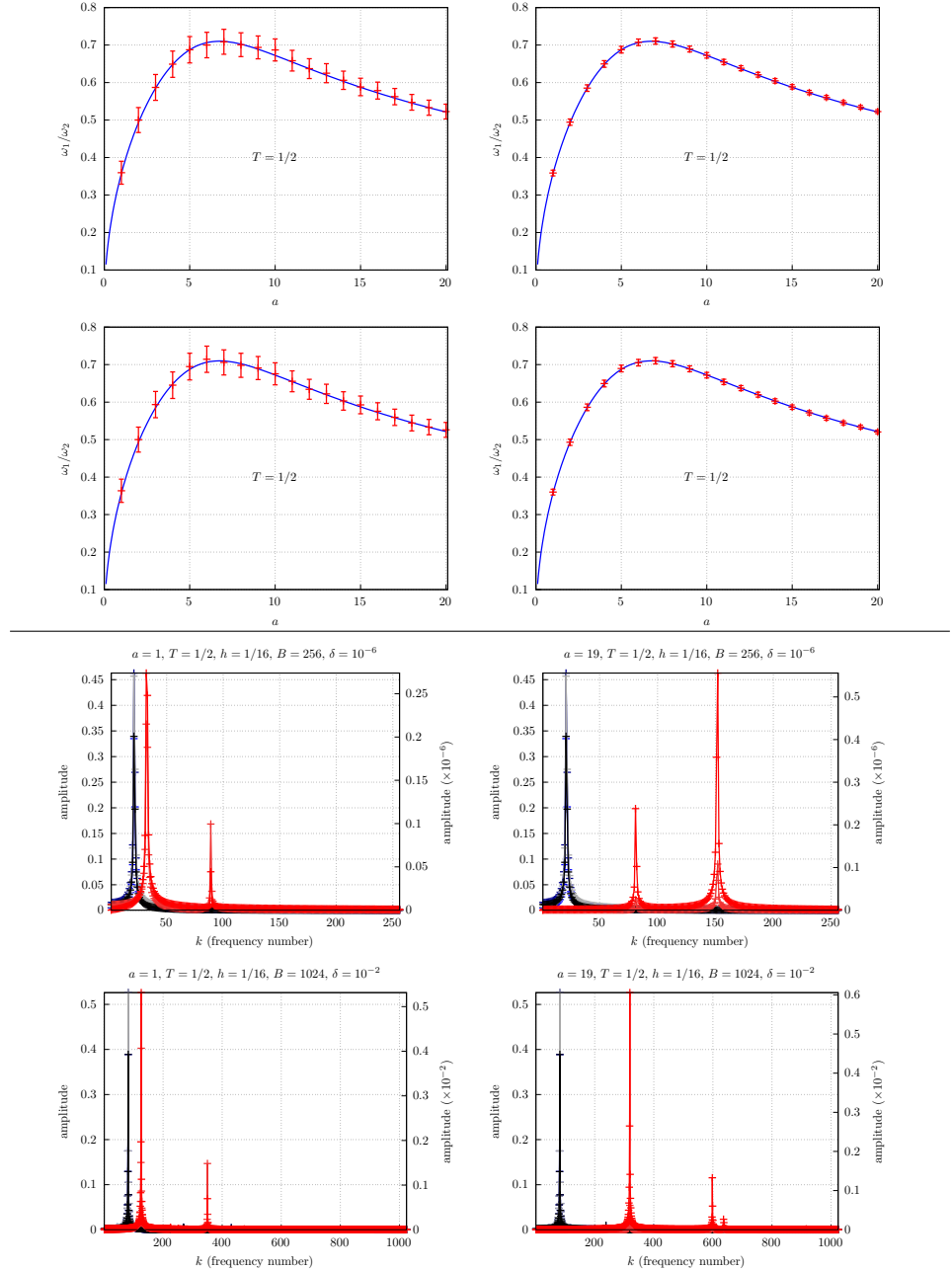

  \centering
  \caption{{\textbf{Nos{é} Thermostat}}. (TL) The frequency ratio
      function $η=ω_1/ω_2$ for the Lennard-Jones 12-6 potential
      $v(r) = r^{-12}-r^{-6}$ as a function of the parameter $a$ for
      the Nos{é} thermostat at temperature $T=1/2$ (in blue) and the
      ``empirical'' frequency ratio as measured by a
      numerically-integrated orbit segment on the interval $[0,B]$,
      with $B=2^8$, stepsize $=2^{-4}$, and a displacement
      $δ = 10^{-6}$ from the relative equilibrium. The error bars are
      (conservative) estimates of the uncertainty in the estimated
      ratios whose height is $(k_1+1)/(k_2-1)-(k_1-1)/(k_2+1)$ when
      $ω_i = 2 π k_i /N$ and $N$ is the sample size. (TR) As in (TL)
      with $B=2^{10}$; the decrease in uncertainty is notable. (ML) As
      in (TL) with $δ = 10^{-2}$. (MR) As in (TR) with $δ =
      10^{-2}$. (BL+R) Selected amplitude vs. frequency number of the
      Fourier transform of $x$ and $p_x$ (dark grey and black, left
      axis); $y$ and $p_y$ (dark blue and blue, left axis) and $s$ and
      $S$ (light red and red, right axis). The mean value of $s$ is
      ignored to highlight the oscillatory modes.}
  \label{fig:nhnd-freq-resp-256}
  \ltxfigure{nhnd/nhnd-freq-resp-256-mac.tex}{0.45\textwidth}{!}
  \ltxfigure{nhnd/nhnd-freq-resp-1024-mac.tex}{0.45\textwidth}{!}
  \ltxfigure{nhnd/nhnd-freq-resp-256-eps0_01-mac.tex}{0.45\textwidth}{!}
  \ltxfigure{nhnd/nhnd-freq-resp-1024-eps0_01-mac.tex}{0.45\textwidth}{!}\\\hrule$ $\\
  \ltxfigure{nhnd/nhnd-freq-ratios-a1-mac.tex}{0.45\textwidth}{!}
  \ltxfigure{nhnd/nhnd-freq-ratios-a19-mac.tex}{0.45\textwidth}{!}
  \ltxfigure{nhnd/nhnd-freq-ratios-1024-a1-mac.tex}{0.45\textwidth}{!}
  \ltxfigure{nhnd/nhnd-freq-ratios-1024-a19-mac.tex}{0.45\textwidth}{!}
\end{figure}

\begin{figure}[p]
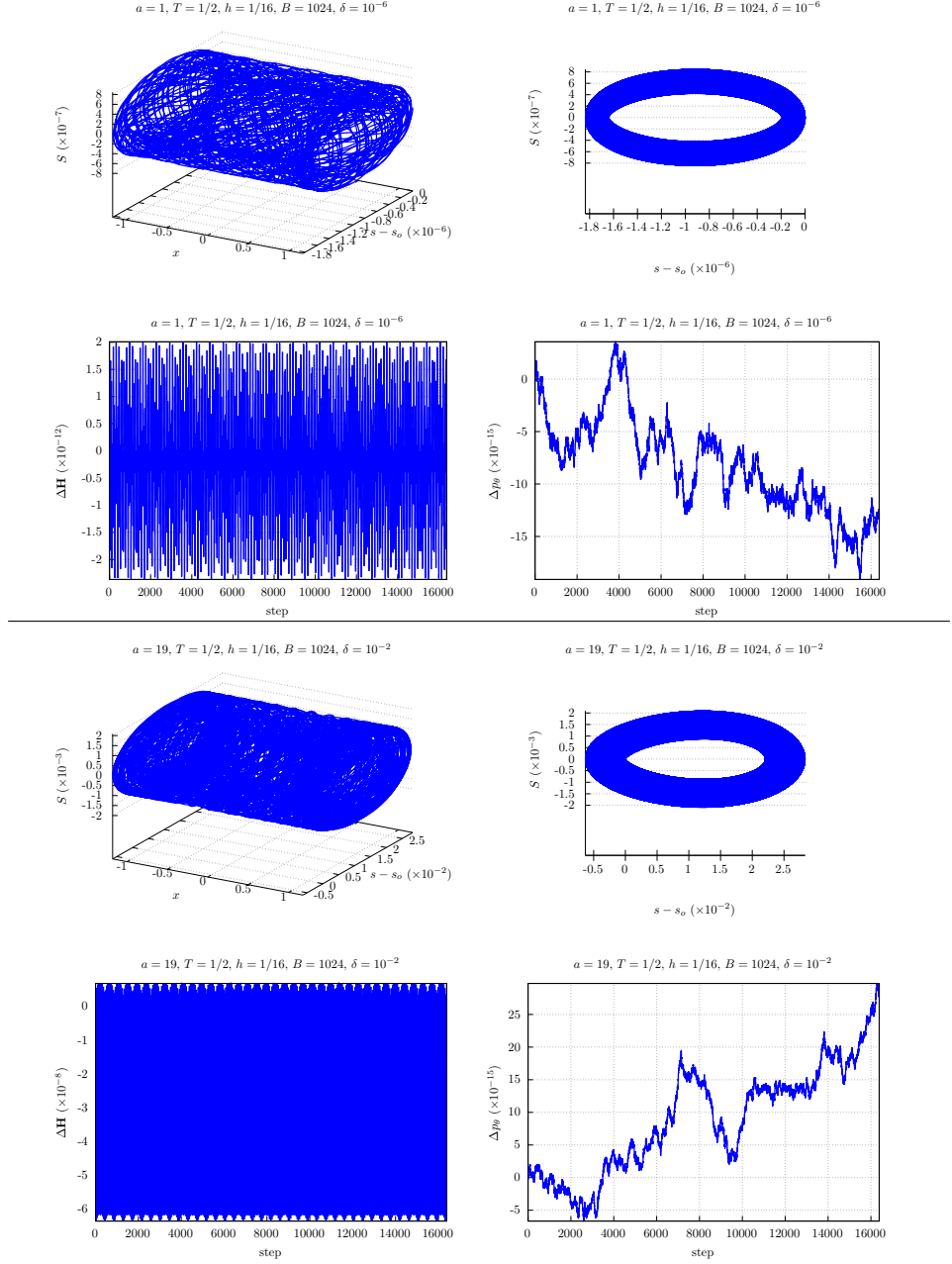

  \caption{{\textbf{Nos{é} Thermostat}}. Top panel: (TL) Projection of
    an orbit segment onto the $s\-S\-x$ $3$-space seen from an oblique
    angle; (TR) The same orbit segment projected onto the $s\-S$
    plane. (BL) Change in energy from its initial value. (BR) Change
    in angular momentum from its initial value.\\ Bottom panel: Same
    as for top panel with different value of $a$ and $δ$.}
  \label{fig:nhnd-3d}
  \ltxfigure{nhnd/nhnd-3d-60-30-a1_0-T0_5-logdelta-6-stepsize0_0625-time1024_0-mac.tex}{0.45\textwidth}{!}             
  \ltxfigure{nhnd/nhnd-3d-90-90-a1_0-T0_5-logdelta-6-stepsize0_0625-time1024_0-mac.tex}{0.45\textwidth}{!}             
  \ltxfigure{nhnd/nhnd-3d-delta-H-a1_0-T0_5-logdelta-6-stepsize0_0625-time1024_0-mac.tex}{0.45\textwidth}{!}             
  \ltxfigure{nhnd/nhnd-3d-delta-ptheta-a1_0-T0_5-logdelta-6-stepsize0_0625-time1024_0-mac.tex}{0.45\textwidth}{!}\\\hrule$ $\\
  \ltxfigure{nhnd/nhnd-3d-60-30-a19_0-T0_5-logdelta-2-stepsize0_0625-time1024_0-mac.tex}{0.45\textwidth}{!}             
  \ltxfigure{nhnd/nhnd-3d-90-90-a19_0-T0_5-logdelta-2-stepsize0_0625-time1024_0-mac.tex}{0.45\textwidth}{!}             
  \ltxfigure{nhnd/nhnd-3d-delta-H-a19_0-T0_5-logdelta-2-stepsize0_0625-time1024_0-mac.tex}{0.45\textwidth}{!}             
  \ltxfigure{nhnd/nhnd-3d-delta-ptheta-a19_0-T0_5-logdelta-2-stepsize0_0625-time1024_0-mac.tex}{0.45\textwidth}{!}
\end{figure}

\begin{figure}[p]
  \centering
  \caption{{\textbf{Logistic Thermostat}}. For a description, see figure~\ref{fig:nhnd-freq-resp-256}}
  \label{fig:nhnd-freq-resp-256-logistic}
  \ltxfigure{nhnd/nhnd-freq-resp-256-logistic-mac.tex}{0.45\textwidth}{!}
  \ltxfigure{nhnd/nhnd-freq-resp-1024-logistic-mac.tex}{0.45\textwidth}{!}
  \ltxfigure{nhnd/nhnd-freq-resp-256-eps0_01-logistic-mac.tex}{0.45\textwidth}{!}
  \ltxfigure{nhnd/nhnd-freq-resp-1024-eps0_01-logistic-mac.tex}{0.45\textwidth}{!}\\\hrule$ $\\
  \ltxfigure{nhnd/nhnd-freq-ratios-a1-logistic-mac.tex}{0.45\textwidth}{!}
  \ltxfigure{nhnd/nhnd-freq-ratios-a19-logistic-mac.tex}{0.45\textwidth}{!}
  \ltxfigure{nhnd/nhnd-freq-ratios-1024-a1-logistic-mac.tex}{0.45\textwidth}{!}
  \ltxfigure{nhnd/nhnd-freq-ratios-1024-a19-logistic-mac.tex}{0.45\textwidth}{!}
\end{figure}

\begin{figure}[p]
  \caption{\textbf{Logistic Thermostat}. For a description, see figure~\ref{fig:nhnd-3d}.}
  \label{fig:nhnd-3d-logistic}
  \ltxfigure{nhnd/nhnd-3d-60-30-a1_0-T0_5-logdelta-6-stepsize0_0625-time1024_0-logistic-mac.tex}{0.45\textwidth}{!}             
  \ltxfigure{nhnd/nhnd-3d-90-90-a1_0-T0_5-logdelta-6-stepsize0_0625-time1024_0-logistic-mac.tex}{0.45\textwidth}{!}             
  \ltxfigure{nhnd/nhnd-3d-delta-H-a1_0-T0_5-logdelta-6-stepsize0_0625-time1024_0-logistic-mac.tex}{0.45\textwidth}{!}             
  \ltxfigure{nhnd/nhnd-3d-delta-ptheta-a1_0-T0_5-logdelta-6-stepsize0_0625-time1024_0-logistic-mac.tex}{0.45\textwidth}{!}\\\hrule$ $\\
  \ltxfigure{nhnd/nhnd-3d-60-30-a19_0-T0_5-logdelta-2-stepsize0_0625-time1024_0-logistic-mac.tex}{0.45\textwidth}{!}             
  \ltxfigure{nhnd/nhnd-3d-90-90-a19_0-T0_5-logdelta-2-stepsize0_0625-time1024_0-logistic-mac.tex}{0.45\textwidth}{!}             
  \ltxfigure{nhnd/nhnd-3d-delta-H-a19_0-T0_5-logdelta-2-stepsize0_0625-time1024_0-logistic-mac.tex}{0.45\textwidth}{!}             
  \ltxfigure{nhnd/nhnd-3d-delta-ptheta-a19_0-T0_5-logdelta-2-stepsize0_0625-time1024_0-logistic-mac.tex}{0.45\textwidth}{!}
\end{figure}

\subsection{The case of surfaces of revolution}
\label{sec:surf-rev}

Let us study the more general class of surfaces with a rotational
symmetry. A standard construction of a surface of revolution, $M$, is
to fix a unit-speed ``profile'' curve $γ(ξ) = (r(ξ),0,z(ξ))$ in the
$x$--$z$ plane and to rotate that curve about the $z$-axis. A
rotationally-invariant mechanical hamiltonian $H : \cotangent M → \R$
has the form
\begin{equation}
  \label{eq:llm-surf-rot-inv-ham}
  H(ξ,p_{ξ},θ,p_{θ}) = ½ \left( p_{ξ}^2 + (p_{θ}/r(ξ))^2 \right) + w(ξ),
\end{equation}
where $θ$ is the angle of rotation and $w : M → \R$ is a
rotationally-invariant potential energy.

\begin{remark}
  \label{rem:llm-surf-rot}
  The simplest examples of profile curves and the associated surfaces
  of revolution are:
  \begin{equation}
    γ(ξ) = \left\{
      \begin{aligned}
        (c,0,ξ),               & \quad & ξ ∈ \R,      & \quad & \textrm{a cylinder of radius $c>0$;} \\
        (ξ,0,0),               & \quad & ξ > 0,       & \quad & \textrm{the $x$--$y$ plane;}         \\
        (C+\cos(ξ),0,\sin(ξ)), & \quad & ξ ∈ [0,2 π], & \quad & \textrm{a torus when $C>1$;}         \\
        (\sin(ξ),0,-\cos(ξ)),  & \quad & ξ ∈ [0,2 π], & \quad & \textrm{the unit sphere.}
      \end{aligned}
      \right.
  \end{equation}
\end{remark}

If $r'(ξ) ≠ 0$ for $ξ$ in an interval $K$, then there is the inverse
function $ξ(r)$ defined on the interval $J=r(K)$. One can equally use
$(r,θ)$ as a local coordinate system on $M$, in which case the
mechanical hamiltonian is transformed to:
\begin{equation}
  \label{eq:llm-surf-rot-inv-ham-r}
  H(r,p_{r},θ,p_{θ}) = ½ \left( \left( c(r) p_{r} \right)^2 + (p_{θ}/r)^2 \right) + v(r),
\end{equation}
where $v(r) = w(ξ(r))$ and $c(r) = 1/ξ'(r)$. The thermostated
hamiltonian $𝐇$ is
\begin{equation}
  \label{eq:llm-surf-rot-inv-therm}
  𝐇(r,p_{r},θ,p_{θ},s,S) = ½ \left( \left( c(r) p_{r} \right)^2 + (p_{θ}/r)^2 \right)/s^2 + v(r) + F(s,aS)/a + T \ln s,
\end{equation}
i.e. the planar case~\eqref{eq:llm-rot-inv-therm-r2} has the same form with $c ≡ 1$.

It follows that lemma~\ref{lem:llm-rot-inv-therm-cov} holds verbatim,
while lemma~\ref{lem:llm-rot-inv-d2g-red} holds with one change: in
the expression for the coefficient $A$, $1/(r s_o)^2$ becomes
$(c(r) / r s_o)^2$ (recall that $r=r_o$ in that lemma). Finally,
lemma~\ref{lem:llm-rot-inv-r-non-deg} holds for the frequency ratio
function $η$, in this case, too.

\subsection{Numerical calculations}
\label{sec:num-calc}

To illustrate Theorem~\ref{thm:llm-rot-inv} and
Corollary~\ref{cor:llm-rot-inv},
figures~\ref{fig:nhnd-freq-resp-256}--\ref{fig:nhnd-3d-logistic}
display a panel of data obtained by integrating the thermostated
planar mechanical system with a Lennard-Jones (12,6) potential at a
temperature of $T=1/2$ and varying values of $a$. The $4$-th order
Candy--Rozmus--Forest--Ruth algorithm is
utilized~\cite{CANDY1991230,FOREST1990105}. This technique is based on
splitting the hamiltonian $𝐇$ into $𝐇_1 + 𝐇_2$ where each $𝐇_i$ is
trivially integrable. Because the kinetic energy is euclidean, this is
accomplished by treating $s$ as a momentum variable and $S$ as a
configuration variable. Interestingly, the proof of the normal-form
Lemma used a similar trick~(see eq. \ref{eq:normal-form-gen-fun-ext}).

\section{Properly Degenerate KAM Theory}
\label{sec:pdgen-kam}

Arnol'd~\cite{MR0170705}, in his attempt to prove the stability of the
$n$-body problem, formulated an important extension of his work on the
stability of quasi-periodic motions~\cite{MR0163025}. In that work,
one considers a hamiltonian on $\T^m × D^m × \R^{2l}$ of the form
\begin{equation}
  \label{eq:h-arnold}
  H_{ε}(θ,I,x,y) = H_0(I) + ε H_1(θ,I,x,y;ε),
\end{equation}
where the perturbation $H_1$ itself is decomposed as
\begin{equation}
  \label{eq:h1-arnold}
  H_1(θ,I,x,y) = P_1(I;ε) + \sum_{|k| ≤ d} α_{k}(I;ε) J^k + O(|x,y|^{2d+1}).
\end{equation}
The coordinates $x_i,y_i$ are canonically conjugate,
$J_i = ½ \left( x_i^2 + y_i^2 \right)$ for $i=1,\ldots,l$ and
$J^k = J_1^{k_1} × \cdots × J_l^{k_l}$. The decomposition of $H_1$ is
obtained by averaging over the fast variables $θ$ to a sufficiently
high degree and then computing the Birkhoff polynomial of the
resulting function of $(x,y)$ (parameterized by $I$).

Arnol'd used the case where $d=3$, while Chierchia \& Pinzari obtain
Arnol'd's results with only $d=2$~\cite{MR2684064,MR2836051}. On the
other hand, Chierchia \& Pusateri prove the following theorem for
$d=1$ (see~\cite{MR2104595} for the $C^{∞}$ case):

\begin{theorem}
  \label{thm:modified-chierchia-pusateri}
  Assume that the real-analytic hamiltonian $H_{ε}$ as in
  \eqref{eq:h-arnold} and \eqref{eq:h1-arnold}, which depends $C^1$ on
  $ε$, has a re-scaled frequency map at $ε=0$,
  $$ Ω(I) = (\D{H}_0(I), α_1(I), \ldots, α_l(I)), $$
  that is R-non-degenerate. Then, for all $ε$ sufficiently small,
  there exists a positive-measure set of phase space that belongs to
  $H_{ε}$-invariant Lagrangian tori. These tori are $O(ε)$-close to
  the Lagrangian tori $\set{I=const., J=O(ε)}$. The flow on each such
  torus is quasi-periodic with Diophantine frequencies.
\end{theorem}

Note that \cite[Theorem 4]{MR2505319} as stated is not the theorem
used to prove the existence of real-analytic KAM tori for the spatial
$n$-body system in that paper--see \cite[p. 870]{MR2505319}. Instead,
the authors use theorem~\ref{thm:modified-chierchia-pusateri} which
allows for the hamiltonian and its decomposition to depend in
non-trivial ways on $ε$--but it must be at least continuous in $ε$. It
is this theorem that is needed to prove Theorem~\ref{thm:main-thm-1}
of the present paper.

\section{Conclusions}
\label{sec:conclusions}

This paper has demonstrated that, in a weakly-coupled regime, the
generalized, variable-mass thermostats of order $2$--including the
Nos{é}, logistic and Winkler thermostats--cannot force most integrable
systems to sample from the Gibbs-Boltzmann distribution. It has also
shown that for some integrable systems (rotationally-invariant
mechanical systems on surfaces), these same thermostats {\em never}
force the system to be ergodic (at least for a full-measure set of
coupling parameters/thermostat masses).

A number of questions remain, though. Here are a select few.

\subsubsection*{1. Effective bounds for $Q_o$:}
Theorem~\ref{thm:main-thm-1} proves the existence of some positive
lower bound $Q_o$ for the thermostat mass (equivalently, a positive
upper bound for the coupling coefficient $a$) beyond which there are
positive-measure sets of KAM tori for the thermostated system. The
theoretical values for the bounds are generally incredibly small--see,
for example, the discussion in \cite[Section 4.4]{MR3467671}. It is
desirable to have a better understanding of this bound, if only for
some particular thermostats and hamiltonians.

\subsubsection*{2. The Thermostated Harmonic Oscillator:}
Theorem~\ref{thm:main-thm-1} does not apply to the 1-d harmonic
oscillator coupled to a generalized, variable-mass thermostat of order
$2$ (remark~\ref{rem:normal-form-near-T-example-contd}). Of course,
Legoll, Luskin \& Moeckel prove the existence of invariant tori in a
neighbourhood of the thermostatic equilibrium set, for all $a$
sufficiently small, by different means~\cite{MR2299758}. On the other
hand, there is an abundance of numerical evidence that suggest KAM
tori persist for $a \cong 1$. Is it true that the ``perturbation''
term $𝐏_{ε}$ in \eqref{eq:normal-form-near-T-h} contains terms that
stabilize the system, even for $ε \cong 1$?

\subsubsection*{3. Extensions of Theorem~\ref{thm:main-thm-1}:}
There are several directions to extend the theorem. Beyond the
previous point, it is desirable to have a generally effective means to
determine if the theorem applies to a particular thermostated
hamiltonian. In addition, order-$2n$ single thermostats appear in the
literature, so it is desirable to extend the theorem to encompass such
thermostats~\cite{TS2016,PhysRevE.75.040102}. A further direction to
extend the theorem is to reversible thermostats that do not
necessarily have a hamiltonian reformulation. The thermostat of
Kusnesov, Bulgac \& Bauer is one such example, while the generalized
Nos{é}-Hoover is a
second~\cite{MR1079786,MR1150101,SprottJulienClinton2014Hcat}.

\section*{Acknowledgments}
\label{sec:ack}

Computations in this paper have been done with Maxima
CAS~\cite{maximaCAS}.

This research has been partially supported by the Natural Science and
Engineering Research Council of Canada grant 320 852.

\bibliographystyle{siam}
\bibliography{nhnd-references}
\end{document}